\documentclass[final,leqno]{siamltex}
\usepackage{amsmath,amssymb,graphicx,subfigure}
\usepackage{hyperref}
\usepackage{doi}
\usepackage{color}
\usepackage{float}

\newcommand{\maxent}{MaxEnt}
\newcommand{\e}{\mathrm{e}}

\newcommand{\dd}{\,\mathrm{d}}
\newcommand{\R}{\mathbb{R}}
\newcommand{\ie}{\emph{i.e.}}
\newcommand{\Er}[1]{E^{(#1)}}
\newcommand{\Be}{\mathrm{Be}}
\def\E{\operatorname{E}}
\def\var{\operatorname{Var}}
\def\co{\operatorname{co}}

\def\interior{\operatorname{int}}

\def\argmin{\operatornamewithlimits{arg\,min}}
\def\set#1{\left\{ #1 \right\}}
\def\abs#1{\left\lvert #1 \right\rvert}
\def\norm#1{\left\lVert #1 \right\rVert}
\def\ip#1{\left\langle #1 \right\rangle}
\newtheorem{thm}{\bf Theorem}
\newtheorem{lem}[thm]{\bf Lemma}

\newtheorem{assump}{\bf Assumption}
\numberwithin{equation}{section}

\title{Discretizing Distributions with Exact Moments: Error Estimate and Convergence Analysis\thanks{\today}}
\author{Ken'ichiro Tanaka\thanks{Corresponding author: School of Systems Information Science, Future University Hakodate (\nobreak{\href{mailto:ketanaka@fun.ac.jp}{\tt ketanaka@fun.ac.jp}}).} 
\and Alexis Akira Toda\thanks{Department of Economics, University of California San Diego (\href{mailto:atoda@ucsd.edu}{\tt atoda@ucsd.edu}).}}

\begin{document}

\maketitle

\begin{abstract}
The maximum entropy principle is a powerful tool for solving underdetermined inverse problems.  This paper considers the problem of discretizing a continuous distribution, which arises in various applied fields.  We obtain the approximating distribution by minimizing the Kullback-Leibler information (relative entropy) of the unknown discrete distribution relative to an initial discretization based on a quadrature formula subject to some moment constraints.  We study the theoretical error bound and the convergence of this approximation method as the number of discrete points increases.  We prove that (i) the theoretical error bound of the approximate expectation of any bounded continuous function has at most the same order as the quadrature formula we start with, and (ii) the approximate discrete distribution weakly converges to the given continuous distribution.  Moreover, we present some numerical examples that show the advantage of the method and apply to numerically solving an optimal portfolio problem.
\end{abstract}



%

\pagestyle{myheadings}
\thispagestyle{plain}
\markboth{KEN'ICHIRO TANAKA AND ALEXIS AKIRA TODA}{DISCRETIZING DISTRIBUTIONS WITH EXACT MOMENTS}

\section{Introduction}\label{sec:intro}
This paper has two goals.  First, we propose a numerical method to approximate continuous probability distributions by discrete ones.  Second, we study the convergence of the method and derive error estimates.  Discretizing continuous distributions is important in applied numerical analysis \cite{keefer-bodily1983,miller-rice1983}.  To motivate the problem, we list a few concrete examples, many of which come from economics.

\paragraph{Optimal portfolio problem}
Suppose that there are $J$ assets indexed by $j=1,\dots,J$.  Asset $j$ has gross return $R_j$ (which is a random variable), which means that a dollar invested in asset $j$ will give a total return of $R_j$ dollars over the investment horizon.  Let $\theta_j$ be the fraction of an investor's wealth invested in asset $j$ (so $\sum_{j=1}^J\theta_j=1$) and $\theta=(\theta_1,\dots,\theta_J)$ be the portfolio.  Then the gross return on the portfolio is the weighted average of each asset return, $R(\theta):=\sum_{j=1}^JR_j\theta_j$.  Assume that the investor wishes to maximize the risk-adjusted expected returns $\E\left[\frac{1}{1-\gamma}R(\theta)^{1-\gamma}\right]$, where $\gamma>0$ is the degree of relative risk aversion.\footnote{Readers unfamiliar with economic concepts need not worry here.  The point is that we want to maximize an expectation that is a function of some parameter.  The case $\gamma=1$ corresponds to the log utility $\E[\log R(\theta)]$.}  Since in general this expectation has no closed-form expression in the parameter $\theta$, in order to maximize it we need to carry out a numerical integration.  This problem is equivalent to finding an approximate discrete distribution of the returns $(R_1,\dots,R_J)$.

\paragraph{Optimal consumption-portfolio problem}
In the above example the portfolio choice was a one time decision.  But we can consider a dynamic version, where the investor chooses the optimal consumption and portfolio over time (Samuelson \cite{samuelson1969} and Merton \cite{merton1971} are classic examples that admit closed-form solutions.  Almost all practical problems, however, admit no closed-form solutions).  Since we need to numerically solve a portfolio problem for each time period, we face one more layer of complexity.

\paragraph{General equilibrium problem}
Instead of solving the optimization problem of a single investor, we can consider a model of the whole economy, and might want to determine the asset prices that make demand equal to supply.  This is called a general equilibrium problem in economics.  Since we need to solve a dynamic optimal portfolio problem given asset prices, and we need to find asset prices that clear markets, we face yet another layer of complexity.  Such problems are especially computationally intensive \cite{huggett1993,aiyagari1994,krusell-smith1998}, and it is important to discretize a continuous distribution with only a small number of support points in order to mitigate the `curse of dimensionality'.

\paragraph{Discretizing stochastic processes}
In many fields, including decision analysis, economic modeling, and option pricing, it is necessary to discretize stochastic processes such as diffusions or autoregressive processes to generate scenario trees \cite{hoyland-wallace2001,pflug2001} or finite state Markov chain approximations \cite{tauchen1986-EL,tauchen-hussey1991,gospodinov-lkhagvasuren2014}.

The above examples have the following common feature.  The researcher is given a continuous density $f$ (transition densities in the case of a stochastic processes), which comes from either some theoretical model or data (say the kernel density estimation).  The density $f$ is used to solve a complicated model.  In order to reduce the complexity of the problem, the researcher wishes to discretize the density.

Since the density is ultimately used to compute expectations, a natural idea is to start with some quadrature formula
\begin{equation}
\E[g(X)] = \int_{\R^{K}} g(x)f(x)\dd x
\approx
\sum_{i=1}^{I_M} w_{i, M}g(x_{i,M})f(x_{i,M}),
\label{eq:SimpleQuadApprox}
\end{equation}
where $M$ is an index of the quadrature formula (for example, the grid size in each dimension), $I_M$ is the number of integration points $\set{x_{i,M}}$, and $w_{i,M}$ is the weight on the point $x_{i,M}$.  Then the probability mass function $q_{i,M}\propto w_{i,M}f(x_{i,M})$ is a valid discrete approximation of the continuous density $f$.  Popular quadrature formulas---such as the Newton-Cotes type or the Gauss type formulas (see \cite{davis-rabinowitz1984} for a standard textbook treatment)---may not be suitable for our purpose, however.  First, the choice of the quadrature formula automatically determines the discrete points $\set{x_{i,M}}$.  However, in practice the researcher often wishes to choose the set of points $D\subset \R^K$ at his or her will, depending on the specific application.  Second, for a general distribution $f$, the quadrature formula \eqref{eq:SimpleQuadApprox} may not even match low order moments such as the mean or the variance, \ie, 
$$\E[X^l] = \int_{-\infty}^\infty x^l f(x)\dd x\neq \sum_{i=1}^{I_M} w_{i, M}x_{i,M}^lf(x_{i,M})$$
for $l=1,2$ (in the one-dimensional case).\footnote{Matching moments of a single density with a quadrature formula is not hard.  The difficulty arises when discretizing a stochastic process.  If we approximate a stochastic process by a finite-state Markov chain, when the process hits the lower or upper boundary of the support points, the conditional moments under the true transition density and the approximate probabilities implied by the quadrature formula will be vastly different, making the approximation poor.}  Although the Newton-Cotes and the Gauss formulas (in the one-dimensional case) match low order moments exactly, the multi-dimensional case is not trivial.  This is a major disadvantage because matching moments is known to improve the accuracy of the analysis \cite{smith1993}.  Third, some quadrature formulas have \emph{negative} weights $w_{i,M}$ when the order is high.  Since we want to obtain probabilities, which are necessarily nonnegative, such formulas are inappropriate.  Thus it is desirable to obtain a discretization method that (i) has a flexible choice of the integration points, (ii) matches at least low order moments, and (iii) always assigns positive weights.

Several methods for discrete approximations of continuous distributions have been proposed in the literature.  Given a continuous probability distribution, Tauchen \cite{tauchen1986-EL} and Adda and Cooper \cite{adda-cooper2003} adopt simple partitions of the support of the distribution and assign the true probability to a representative point of each partitioned domain.  Although their methods are intuitive, simple, and work in any dimension, their methods are not so accurate, for they generate discrete distributions with only approximate moments.  Miller and Rice \cite{miller-rice1983}, Tauchen and Hussey \cite{tauchen-hussey1991}, and Smith \cite{smith1993} discretize the density function using the weights and the points of the Gaussian quadrature, and Devuyst and Preckel \cite{devuyst-preckel2007} consider its generalization to multi-dimensions.  Although their methods are often more accurate and can match prescribed polynomial moments exactly, they do not allow for the restriction $\set{x_{i, M}} \subset D$ and cannot be applied to non-polynomial moments.  Furthermore, the multi-dimensional method by Devuyst and Preckel \cite{devuyst-preckel2007} is computationally intensive and does not have a theoretical guarantee for the existence of the discretization, error bounds, or convergence.

As a remedy, in Tanaka and Toda \cite{tanaka-toda-maxent-approx} we proposed an approximation method based on the maximum entropy principle (\maxent) that exactly matches prescribed moments of the distribution.  Starting with any discretization, we ``fine-tune'' the given probabilities by minimizing the Kullback-Leibler information (relative entropy) of the unknown probabilities subject to some moment constraints.  In that paper we proved the existence and the uniqueness of the solution to the minimization problem, showed that the solution can be easily computed by solving the dual problem, and presented some numerical examples that show that the approximation method is satisfactorily accurate.  The method is computationally very simple and works on any discrete set $D$ of any dimension with any prescribed moments (not necessarily polynomials).  However, up to now the theoretical approximation error and the convergence of this method remain unknown.

This paper gives a theoretical error bound for this approximation method and shows its convergence.  We first evaluate the theoretical error of our proposed method.  It turns out that the order of the theoretical error estimate is at most that of the initial discretization, and actually improves if the integrand is well-approximated by the moment defining function.  Thus our proposed method does not compromise the order of the error at the expense of matching moments.  Second, as a theoretical consequence of the error estimate, we show the weak convergence of the discrete distribution generated by the method to the given continuous distribution.  This means that for any bounded continuous function $g$, the expectation of $g$ under the approximating discrete distribution converges to that under the exact distribution as the number of integration points increases.  This property is practically important because it guarantees that the approximation method never generates a pathological discrete distribution with exact moments which has extremely different probability from the given distribution on some domain, at least when the discrete set is large enough.  In addition, we present some numerical examples (including a numerical solution to an optimal portfolio problem) that show the advantage of our proposed method.

The idea of using the maximum entropy principle to obtain a solution to underdetermined inverse problems (such as the Hausdorff moment problem) is similar to that of Jaynes \cite{jaynes1982} and Mead and Papanicolaou \cite{mead-papanicolaou1984}. 
There is an important distinction, however.  In typical inverse problems, one studies the convergence of the approximating solution to the true one when \emph{the number of moment constraints} tends to infinity.  In contrast, in this paper we study the convergence when \emph{the number of approximating points} tends to infinity, fixing the moments.  Thus the two problems are quite different.  The literature on the foundations, implementations, and the applications of maximum entropy methods is immense: any literature review is necessarily partial.  The maximum entropy principle (as an inference method in general, not necessarily restricted to physics) was proposed by Jaynes \cite{jaynes1957a}.  For axiomatic approaches, see Shore and Johnson \cite{shore-johnson1980}, Jaynes \cite{jaynes2003}, Caticha and Giffin \cite{caticha-giffin2006-AIP}, and Knuth and Skilling \cite{knuth-skilling2012}.  For the relation to Bayesian inference, see Van Campenhout and Cover \cite{vancampenhout-cover1981} and Csisz{\'a}r \cite{csiszar1984}.  For the duality theory of entropy maximization, see \cite{borwein-lewis1991,DHLN1992,hauck-levermore-tits2008}.  For numerical algorithms for computing maximum entropy densities, see \cite{abramov2007,abramov2010,AHOT2014}.  Budi\v{s}i\'c and Putinar \cite{budisic-putinar2012} study the opposite problem of ours, namely transforming a probability measure to one that is absolutely continuous with respect to the Lebesgue measure.  Applications of maximum entropy methods can be found in economics \cite{foley1994,toda2010,toda-BGE}, statistics and econometrics \cite{barron-sheu1991,kitamura-stutzer1997,wu2003}, finance \cite{stutzer1995,stutzer1996,buchen-kelly1996}, among many other fields.


\section{The approximation method}\label{sec:OurMethod}
This section reviews the discrete approximation method of continuous distributions proposed in \cite{tanaka-toda-maxent-approx}. 

Let $f$ be a probability density function on $\R^K$ with some generalized moments 
\begin{align}
\bar{T}=\int_{\R^{K}} f(x)T(x)\dd x,
\label{eq:given_moments}
\end{align}
where $T:\R^K\to \R^L$ is a continuous function.  (Below, we sometimes refer to this function as the ``moment defining function''.)
For instance, if we are interested in the first and second polynomial moments (\ie, mean and variance), $T$ becomes
$$T(x)=(x_1,\dots,x_K,x_1^2\dots,x_kx_l,\dots,x_K^2).$$
In this case, we have $K$ expectations, $K$ variances, and $\frac{K(K-1)}{2}$ covariances.  Hence the total number of moment constraints (the dimension of the range space of $T$) is
\begin{equation}
L=K+K+\frac{K(K-1)}{2}=\frac{K(K+3)}{2}.\label{eq:nummoments}
\end{equation}
In general, the components of $T(x)$ need not be polynomials.

Moreover, for each positive integer $M$, assume that 
a finite discrete set
$$D_{M} = \set{x_{i, M} \mid i = 1, \ldots, I_M} \subset \R^K$$
is given, where $I_M$ is the number of discrete points.  An example of $D_{M}$ is the lattice 
$$D_{M} = \set{ (m_{1} h, m_{2} h, \ldots, m_{K} h) \mid m_{1}, m_{2} \ldots, m_{K} = 0, \pm 1, \dots, \pm M},$$
where $h>0$ is the grid size, in which case $I_M = (2M+1)^{K}$.  Our aim is to find a discrete probability distribution 
$$P_{M} = \set{ p_{i,M} \mid x_{i, M}\in D_{M}}$$
on $D_{M}$ with exact moments $\bar{T}$ that approximates $f$ (in the sense of the weak topology, that is, convergence in distribution).

To match the moments $\bar{T}$ with $P_{M} = \set{ p_{i,M} \mid x_{i, M}\in D_{M}}$, it suffices to assign probabilities $\set{p_{i,M}}$ such that
\begin{align}
\sum_{i=1}^{I_M} p_{i,M}T(x_{i,M})=\bar{T}.
\label{eq:P_const}
\end{align}
Note that the solution to this equation is generally underdetermined because the number of unknowns---$p_{i,M}$'s---of which there are $I_M$, is typically much larger than the number of equations (moments), $L+1$.\footnote{The ``$+1$'' comes from accounting the probabilities $\sum_{i=1}^{I_M} p_{i,M}=1$.}
To obtain $P_{M}$ that approximates $f$ and satisfies \eqref{eq:P_const}, we first choose an arbitrary discretization of $f$ with \emph{not necessarily exact moments}, which we denote by $Q_M=\set{q_{i,M}}$.  For example, if we already have a numerical quadrature formula 
\begin{equation}
\int_{\R^{K}} f(x)g(x)\dd x\approx \sum_{i=1}^{I_M} w_{i,M}f(x_{i,M})g(x_{i,M})\label{eq:quad}
\end{equation}
with positive weights $w_{i,M}$ $(i=1,2,\ldots, I_M)$, where $g$ is an arbitrary function that we want to compute the expectation with respect to the density $f$, then it is natural to set $q_{i,M}$ proportional to $w_{i,M}f(x_{i,M})$, \ie,
$$q_{i,M}=\frac{w_{i,M}f(x_{i,M})}{\sum_{i=1}^{I_M} w_{i,M}f(x_{i,M})}.$$
In the following, we do not address how to choose $Q_M$ (or the quadrature formula \eqref{eq:quad}) but take it as given.

Now the approximation method is defined as follows.  Let $g$ be any bounded continuous function.  Then the approximation error of the expectation under $P_M$ is
\begin{align}
&\abs{\int_{\R^K} f(x)g(x)\dd x-\sum_{i=1}^{I_M} p_{i,M}g(x_{i,M})}\label{eq:triangle}\\
&\le \abs{\int_{\R^K} f(x)g(x)\dd x-\sum_{i=1}^{I_M} q_{i,M}g(x_{i,M})}+\abs{\sum_{i=1}^{I_M} q_{i,M}g(x_{i,M})-\sum_{i=1}^{I_M} p_{i,M}g(x_{i,M})}\notag\\
&\le \abs{\int_{\R^K} f(x)g(x)\dd x-\sum_{i=1}^{I_M} q_{i,M}g(x_{i,M})}+\norm{g}_\infty\sum_{i=1}^{I_M}\abs{p_{i,M}-q_{i,M}},\notag
\end{align}
where $\norm{g}_\infty=\sup_{x\in \R^K}\abs{g(x)}$ is the sup norm.  Since the first term depends only on the initial discretization $Q_M$, we focus on the second term.  By Pinsker's inequality,\footnote{See \cite{csiszar1967,kullback1967} for proofs of Pinsker's inequality and \cite{fedotov-harremoes2003} and references therein for refinements.  The appendix of this paper gives a short proof of Pinsker's inequality.} the second term can be bounded as follows:
\begin{equation}
\sum_{i=1}^{I_M}\abs{p_{i,M}-q_{i,M}}\le \sqrt{2H(P_M;Q_M)},\label{eq:Pinsker}
\end{equation}
where
\begin{equation}
H(P_M;Q_M)=\sum_{i=1}^{I_M} p_{i,M}\log \frac{p_{i,M}}{q_{i,M}}\label{eq:KL}
\end{equation}
is the Kullback-Leibler information \cite{kullback-leibler1951}.

In order to make the approximation error \eqref{eq:triangle} small, given Pinsker's inequality \eqref{eq:Pinsker}, it is quite natural to obtain the approximate discrete distribution $P_{M} = \set{p_{i,M} \mid x_{i, M}\in D_{M}}$ as the solution to the following optimization problem:
\begin{align}
&\min_{\set{p_{i,M}}} \sum_{i=1}^{I_M} p_{i,M}\log \frac{p_{i,M}}{q_{i,M}} \tag{P}\label{eq:P}\\
&\text{subject to}~\sum_{i=1}^{I_M} p_{i,M}T(x_{i, M})=\bar{T},~\sum_{i=1}^{I_M} p_{i,M}=1,~p_{i,M}\ge 0. \notag
\end{align}
The first constraint matches the moments $\int f(x)T(x)\dd x=\bar{T}$ exactly.  $\sum_{i=1}^{I_M} p_{i,M}=1$ and $p_{i,M}\ge 0$ ensure that $\set{p_{i,M}}$ is a probability mass function.  The problem \eqref{eq:P} has a unique solution if $\bar{T} \in \co T(D_{M})$, where $\co T(D_{M})$ is the convex hull of $T(D_{M})$ defined by
\begin{align}
\co T(D_{M})
= \set{ \sum_{i=1}^{I_M} \alpha_{i, M} T(x_{i, M}) \,
\left| \, \sum_{i=1}^{I_M} \alpha_{i, M} =1 \text{ and } \alpha_{i, M} \ge 0 \right. },
\label{eq:convTD}
\end{align}
because in that case the constraint set is nonempty, compact, convex, and the objective function is continuous (by adopting the convention $0\log 0=0$) and strictly convex.

To characterize the solution of \eqref{eq:P}, we consider the Fenchel dual\footnote{See \cite{borwein-lewis1991} for an application of the Fenchel duality to entropy-like minimization problems.} of \eqref{eq:P},
\begin{equation}
\max_{\lambda\in\R^L} 
\left[\ip{\lambda,\bar{T}}-\log \left(\sum_{i=1}^{I_M} q_{i,M}\e^{\ip{ \lambda,T(x_{i,M}) } }\right)\right], \tag{D}\label{eq:D}
\end{equation}
where $\ip{\, \cdot \, , \, \cdot \,}$ denotes the usual inner product in $\R^{L}$.  Tanaka and Toda \cite{tanaka-toda-maxent-approx} show that we can obtain the solution to \eqref{eq:P} as fine-tuned values of $q_{i,M}$, and that the minimum value of \eqref{eq:P} and the maximum value of \eqref{eq:D} coincide.  Although these properties are routine exercises in convex duality theory (see for example \cite{borwein-lewis2006} for a textbook treatment), we present them nevertheless in order to make the paper self-contained.
\begin{thm}\label{thm:tanaka-toda}
Suppose that $\bar{T} \in \interior\co T(D_{M})$, where $\interior$ denotes the interior. Then the following is true.
\begin{enumerate}
\item The dual problem \eqref{eq:D} has a unique solution $\lambda_M$.
\item The probability distribution $P_{M} = \set{ p_{i,M} \mid x_{i,M} \in D_{M} }$ defined by
\begin{equation}
p_{i,M}
=
\frac{q_{i,M}\e^{\ip{ \lambda_{M},T(x_{i,M}) } }}{\sum_{i=1}^{I_M} q_{i,M}\e^{\ip{ \lambda_{M},T(x_{i,M}) } }}
\label{eq:solution_P}
\end{equation}
is the unique solution to \eqref{eq:P}.
\item The duality $H(P_M;Q_M)=\min \eqref{eq:P}=\max \eqref{eq:D}$ holds.
\end{enumerate}
\end{thm}

Our proposed approximation method has two advantages.  The first is the computational simplicity.  The dual problem \eqref{eq:D} is an \emph{unconstrained} optimization problem with typically a small number of unknowns ($L$), whereas the primal problem \eqref{eq:P} is a \emph{constrained} optimization problem with typically a large number of unknowns ($I_M$).  With existing methods such as Gospodinov and Lkhagvasuren \cite{gospodinov-lkhagvasuren2014} that target the first and second moments with a quadratic loss function, the optimization problem remains high dimensional and constrained.  For example, if we discretize a 3-dimensional distribution with 10 discrete points in each dimension and match the mean and the variances, then there will be $10^3=1,000$ unknowns, 9 moment constraints (according to \eqref{eq:nummoments}), and 1,000 nonnegativity constraints.  On the other hand, our dual problem \eqref{eq:D} is unconstrained and involves only 9 variables.  The second advantage is that the resulting weights on the points \eqref{eq:solution_P} are automatically positive, as they should be since they are probabilities.

\section{Error bound and convergence}\label{sec:ErrConv}
Let $g: \R^{K} \to \R$ be a bounded continuous function and
\begin{align}
\Er{Q}_{g, M} 
&=\abs{\int_{\R^{K}} f(x)g(x)\dd x - \sum_{i=1}^{I_M} q_{i,M}g(x_{i,M})},
\label{eq:Qerror}\\
\Er{P}_{g, M} 
&=\abs{\int_{\R^{K}} f(x)g(x)\dd x - \sum_{i=1}^{I_M} p_{i,M}g(x_{i,M})}
\label{eq:Perror}
\end{align}
be the approximation errors under the initial discretization $Q_M$ and $P_M=\set{p_{i,M}}$ in \eqref{eq:solution_P}.  In this section we estimate the error $\Er{P}_{g,M}$ and prove the weak convergence\footnote{For readers unfamiliar with probability theory, a sequence of probability measures $\set{\mu_n}$ is said to \emph{weakly converge} to $\mu$ if $\lim_{n\to\infty}\int g\dd \mu_n=\int g\dd \mu$ for every bounded continuous function $g$.  In particular, by choosing $g$ as an indicator function, we have $\mu_n(B)\to \mu(B)$ for any Borel set $B$, so the probability distribution $\mu_n$ approximates $\mu$.} of $P_{M}=\set{p_{i,M}}$ to $f$, \ie, $\Er{P}_{g, M} \to 0$ as $M \to \infty$ for any $g$.  Using the definition of the errors \eqref{eq:Qerror} and \eqref{eq:Perror}, the triangle inequality \eqref{eq:triangle}, and Pinsker's inequality \eqref{eq:Pinsker}, we obtain the following error estimate:
\begin{equation}
\Er{P}_{g,M}\le \Er{Q}_{g,M}+\norm{g}_\infty \sqrt{2H(P_M;Q_M)},\label{eq:errest1}
\end{equation}
where $H(P_M;Q_M)$ is the Kullback-Leibler information \eqref{eq:KL}.

We consider the error estimate and the convergence analysis under the following two assumptions.  The first assumption states that the moment defining function $T$ has no degenerate components and the moment $\bar{T}$ can also be expressed as an expectation on the discrete set $D_{M}$.

\begin{assump}
\label{assump:int_T}
The components of the moment defining function $T$ are affine independent on $\R^L \cap \mathop{\mathrm{supp}} f$: for any $0\neq (\lambda,\mu) \in \R^L\times\R$, there exists $x\in \mathop{\mathrm{supp}} f$ such that $\ip{\lambda,T(x)} + \mu \neq 0$.
Furthermore, $\bar{T} \in \interior (\co T(D_{M}))$ for any $M$.
\end{assump}

The second assumption concerns the convergence of the initial discretization $Q_M=\set{q_{i,M}}$.

\begin{assump}
\label{assump:int_formula}
The initial discretization weakly converges: for any bounded continuous function $g$ on $\R^{K}$, we have
\begin{equation}
\lim_{M \to \infty}\sum_{i=1}^{I_M} q_{i,M}g(x_{i,M}) = \int_{\R^{K}} f(x)g(x)\dd x. \label{eq:int_gconv}
\end{equation}
Furthermore, the targeted moments converge as well:
\begin{equation}
\lim_{M \to\infty}
\sum_{i=1}^{I_M} q_{i,M}T(x_{i,M})=\int_{\R^{K}} f(x)T(x)\dd x. \label{eq:int_Tconv}
\end{equation}
\end{assump}
Assumption \ref{assump:int_formula} is natural since any discrete approximation should become accurate as we increase the number of grid points.  Under this assumption, we have $\Er{Q}_{g,M}\to 0$ as $M\to\infty$, so by \eqref{eq:errest1} we have $\Er{P}_{g,M}\to 0$ whenever $H(P_M;Q_M)\to 0$ as $M\to \infty$.  Below, we focus on estimating $H(P_M;Q_M)$.

\begin{lem}\label{lem:H_minJ}
Define the function $J_M:\R^L\to \R$ by
\begin{equation}
J_M(\lambda)=\sum_{i=1}^{I_M} q_{i,M}\e^{\ip{\lambda,T(x_{i,M})-\bar{T}}}.\label{eq:Def_J}
\end{equation}
Then
\begin{equation}
H(P_M;Q_M)=-\log \left(\min_{\lambda\in\R^L}J_M(\lambda)\right).\label{eq:H_minJ}
\end{equation}
\end{lem}

\begin{proof}
By Theorem \ref{thm:tanaka-toda} and the definition of $J_M$, we obtain
\begin{align*}
H(P_M;Q_M)&=\max_{\lambda\in\R^L} 
\left[\ip{\lambda,\bar{T}}-\log \left(\sum_{i=1}^{I_M} q_{i,M}\e^{\ip{ \lambda,T(x_{i,M}) } }\right)\right]&&(\because \text{\eqref{eq:D}, Theorem \ref{thm:tanaka-toda}})\\
&=\max_{\lambda\in\R^L} 
\left[-\log \left(\sum_{i=1}^{I_M} q_{i,M}\e^{\ip{ \lambda,T(x_{i,M})-\bar{T}} }\right)\right]&&(\because \sum q_{i,M}=1)\\
&=\max_{\lambda\in\R^L}[-\log (J_M(\lambda))]=-\log\left(\min_{\lambda\in\R^L}J_M(\lambda)\right),&&(\because \eqref{eq:Def_J})
\end{align*}
which is \eqref{eq:H_minJ}.
\end{proof}

By Lemma \ref{lem:H_minJ}, in order to estimate the Kullback-Leibler information $H(P_M;Q_M)$ from above, it suffices to bound $J_M(\lambda)$ from below.  To this end let
\begin{equation}
\Er{Q}_{T,M} = \norm{\int_{\R^{K}} f(x)T(x)\dd x -
\sum_{i=1}^{I_M} q_{i,M}T(x_{i,M})}
\label{eq:Err1_g_T} 
\end{equation}
be the initial approximation error for the moments $T$,
\begin{equation}
C_\alpha=\inf_{\lambda\in\R^L,\norm{\lambda}=1}\int_{\R^K}f(x)\left(\max\set{0,\min\set{\ip{\lambda,T(x)-\bar{T}},\alpha}}\right)^2\dd x
\label{eq:Def_Calpha}
\end{equation}
for any $\alpha>0$, and $C_\infty=\lim_{\alpha\to\infty}C_\alpha$.
The following lemma gives a quadratic lower bound of $J_M$.

\begin{lem}\label{lem:J_quad}
Let Assumptions \ref{assump:int_T} and \ref{assump:int_formula} be satisfied.  Then
\begin{enumerate}
\item $0<C_\alpha\le C_\infty\le \infty$.
\item For any $C$ with $0<C<C_\infty$, there exists a positive integer $M_C$ such that for any $M\ge M_C$, we have
\begin{equation}
J_M(\lambda)\ge 1-\Er{Q}_{T,M}\norm{\lambda}+\frac{1}{2}C\norm{\lambda}^2.\label{eq:J_quad}
\end{equation}
\end{enumerate}
\end{lem}
\begin{proof}
First we prove $0<C_\alpha\le C_\infty\le \infty$.  Since the integrand in \eqref{eq:Def_Calpha} is nonnegative and increasing in $\alpha>0$, clearly $C_\alpha\le C_\infty\le \infty$.  Due to the presence of the $\max$ and $\min$ operators, the integrand is positive if and only if $x\in \supp f$ and $\ip{\lambda,T(x)-\bar{T}}>0$.  Since the components of $T$ are affine independent on $\supp f$ (Assumption \ref{assump:int_T}) and
\begin{align*}
\bar{T}=\int_{\R^K} f(x)T(x)\dd x &\iff \int_{\R^K} f(x)(T(x)-\bar{T})\dd x=0\\
&\iff (\forall \lambda)~\int_{\R^K} f(x)\ip{\lambda,T(x)-\bar{T}}\dd x=0,
\end{align*}
for any $\lambda\neq 0$ there exists a region in $\supp f$ for which $\ip{\lambda,T(x)-\bar{T}}>0$.  Since $T$ is continuous, the integrand in \eqref{eq:Def_Calpha} is positive on a set with positive measure and continuous with respect to $\lambda$, so $C_\alpha>0$.

Next we prove \eqref{eq:J_quad}.  For this purpose, we use the inequality
\begin{equation}
\e^z\ge 1+z+\frac{1}{2}\left(\max\set{0,\min\set{z,a}}\right)^2\label{eq:temp.1}
\end{equation}
for any $z,a\in\R$.  \eqref{eq:temp.1} follows from
$$\max\set{0,\min\set{z,a}}=\begin{cases}
0,&(z\le 0~\text{or}~a\le 0)\\
z,&(0\le z\le a)\\
a,&(0\le a\le z)
\end{cases}$$
$\e^z\ge 1+z$ if $z\le 0$, and $\e^z\ge 1+z+\frac{1}{2}z^2$ if $z\ge 0$.

Let $z=\ip{\lambda,T(x_{i,M})-\bar{T}}$ and $a=\norm{\lambda}\alpha$ for $\alpha>0$ in the inequality \eqref{eq:temp.1}.  Then
$$\e^{\ip{\lambda,T(x_{i,M})-\bar{T}}}\ge 1+\ip{\lambda,T(x_{i,M})-\bar{T}}+\frac{1}{2}\left(\max\set{0,\min\set{\ip{\lambda,T(x_{i,M})-\bar{T}},\norm{\lambda}\alpha}}\right)^2.$$
Multiplying both sides by $q_{i,M}\ge 0$, letting $\lambda^*=\frac{\lambda}{\norm{\lambda}}$, and summing over $i$, we obtain
\begin{equation}
J_M(\lambda)\ge 1+\ip{\lambda,B_M}+\frac{1}{2}C_{M,\alpha}(\lambda^*)\norm{\lambda}^2,\label{eq:temp.2}
\end{equation}
where
\begin{align}
B_M&=\sum_{i=1}^{I_M} q_{i,M}(T(x_{i,M})-\bar{T}),\label{eq:Def_BM}\\
C_{M,\alpha}(\lambda^*)&=\sum_{i=1}^{I_M} q_{i,M}\left(\max\set{0,\min\set{\ip{\lambda^*,T(x_{i,M})-\bar{T}},\alpha}}\right)^2.\label{eq:Def_CMalpha}
\end{align}
Letting $M\to \infty$ in \eqref{eq:Def_CMalpha}, we obtain
\begin{align*}
C_{M,\alpha}(\lambda^*)&\to \int f(x)\left(\max\set{0,\min\set{\ip{\lambda^*,T(x)-\bar{T}},\alpha}}\right)^2\dd x&&(\because \text{Assumption \ref{assump:int_formula}})\\
&\ge C_\alpha. &&(\because \eqref{eq:Def_Calpha})
\end{align*}
Since $\lim_{\alpha\to \infty}C_\alpha=C_\infty>C>0$, we can take $\alpha>0$ large enough such that $C_\alpha>C$.  Since $\lim_{M\to\infty} C_{M,\alpha}(\lambda^*)\ge C_\alpha>C$, we can take $M_C$ such that for any $M\ge M_C$ we have $C_{M,\alpha}(\lambda^*)\ge C$.  Then by \eqref{eq:temp.2} and the Cauchy-Schwarz inequality, we obtain
$$J_M(\lambda)\ge 1-\norm{B_M}\norm{\lambda}+\frac{1}{2}C\norm{\lambda}^2.$$
Since $\norm{B_M}=\Er{Q}_{T,M}$ by \eqref{eq:Err1_g_T}, $\int_{\R^K}f(x)T(x)\dd x=\bar{T}$, and \eqref{eq:Def_BM}, we obtain the conclusion \eqref{eq:J_quad}.
\end{proof}

Combining the above lemmas, we obtain the following estimate of $\Er{P}_{g, M}$.

\begin{thm}
\label{thm:MainErrorEstimate}
Let Assumptions \ref{assump:int_T} and \ref{assump:int_formula} be satisfied and $g$ be a bounded continuous function.  Then, for any $C>0$ be with $0<C<C_\infty=\lim_{\alpha\to\infty}C_\alpha$, there exists a positive integer $M_C$ such that for any $M$ with $M \ge M_C$, we have
\begin{equation}
\Er{P}_{g,M}\le \Er{Q}_{g,M}+\norm{g}_\infty\sqrt{-2 \log \left(1-\frac{(\Er{Q}_{T,M})^2}{2C} \right)},
\label{eq:final_estim}
\end{equation}
where the left-hand side is interpreted as $\infty$ if the inside of logarithm is negative.
\end{thm}

\begin{proof}
For notational simplicity let $E:=\Er{Q}_{T,M}$.  Minimizing the right-hand side of \eqref{eq:J_quad} analytically, we obtain
\begin{equation}
J_M(\lambda)\ge 1-\frac{E^2}{2C}.\label{eq:J_lb}
\end{equation}
If $E^2\ge 2C$, \eqref{eq:J_lb} is not tight since $J_M\ge 0$ always.  In this case, we obtain the trivial estimate $\Er{P}_{g,M}\le \infty$.  Otherwise,
\begin{align*}
\Er{P}_{g,M}&\le \Er{Q}_{g,M}+\norm{g}_\infty \sqrt{-2\log \left(\min_\lambda J_M(\lambda)\right)}&& (\because \eqref{eq:errest1}, \eqref{eq:H_minJ})\\
&\le \Er{Q}_{g,M}+\norm{g}_\infty \sqrt{-2\log(1-E^2/2C)}, &&(\because \eqref{eq:J_lb})
\end{align*}
which is \eqref{eq:final_estim}.
\end{proof}

Note that $\Er{P}_{g, M}$ is bounded by a formula consisting of $\Er{Q}_{g,M}$ and $\Er{Q}_{T,M}$, which are the errors of the initial discretization $Q_M$ for the functions $g$ and $T$.  Since both of them tend to zero as $M \to \infty$ by Assumption \ref{assump:int_formula}, it follows from \eqref{eq:final_estim} and $-\log(1-t)\approx t$ for small $t$ that
\begin{equation}
\Er{P}_{g, M} = \mathrm{O}
\left(
\max\set{ 
\Er{Q}_{g,M},\Er{Q}_{T,M}} \right) \quad (M \to \infty).
\label{eq:final_order}
\end{equation}
The equality \eqref{eq:final_order} shows that the error $E_{g,M}$ is at most of the same order as the error of the initial discretization.  Thus our method does not compromise the order of the error at the expense of matching moments.

Using Theorem \ref{thm:MainErrorEstimate}, we can prove our main result, the weak convergence of the approximating discrete distribution $P_{M}=\set{p_{i,M}}$ to $f$.

\begin{thm}
\label{thm:MainConv}
Let Assumptions \ref{assump:int_T} and \ref{assump:int_formula} be satisfied.
Then, for any bounded continuous function $g$, we have
\begin{align}
\lim_{M \to \infty} \sum_{i=1}^{I_M} p_{i,M}g(x_{i,M}) = \int_{\R^{K}} f(x)g(x)\dd x,
\label{eq:MainConv}
\end{align}
\ie, the discrete distribution $P_{M}$ weakly converges to the exact distribution $f$.
\end{thm}

\begin{proof}
By the definition of $\Er{P}_{g,M}$ in \eqref{eq:Perror}, \eqref{eq:final_order}, and Assumption \ref{assump:int_formula}, we get
$$\abs{\int_{\R^{K}} f(x)g(x)\dd x-\sum_{i=1}^{I_M} p_{i,M}g(x_{i,M})}=\Er{P}_{g,M}=\mathrm{O}\left(\max\set{\Er{Q}_{g,M},\Er{Q}_{T,M}} \right)\to 0$$
as $M\to\infty$, which is \eqref{eq:MainConv}.
\end{proof}

The following theorem gives a tighter error estimate when $\Er{Q}_{T,M}$ is large.

\begin{thm}\label{thm:alt_error}
Let everything be as in Theorem \ref{thm:MainErrorEstimate}.  Then
\begin{equation}
\Er{P}_{g,M}\le \Er{Q}_{g,M}+\norm{g}_\infty\frac{2}{\sqrt{C}}\Er{Q}_{T,M}.\label{eq:alt_estim}
\end{equation}
\end{thm}

\begin{proof}
By \eqref{eq:errest1} it suffices to show
\begin{equation}
H(P_M;Q_M)\le \frac{2}{C}(\Er{Q}_{T,M})^2.\label{eq:H_estim}
\end{equation}
Let $\lambda_M\in\R^L$ be the solution to the dual problem \eqref{eq:D}.  By Theorem \ref{thm:tanaka-toda}, we have
\begin{align}
H(P_M;Q_M)&=-\log \left(\sum_{i=1}^{I_M} q_{i,M}\e^{\ip{\lambda_M,T(x_{i,M})-\bar{T}}}\right)\label{eq:H_ub}\\
&\le -\sum_{i=1}^{I_M}q_{i,M}\log \left(\e^{\ip{\lambda_M,T(x_{i,M})-\bar{T}}}\right)&&(\because \text{$-\log(\cdot)$ is convex})\notag\\
&=-\ip{\lambda_M,\sum_{i=1}^{I_M}q_{i,M}T(x_{i,M})-\bar{T}} && (\because \sum q_{i,M}=1)\notag\\
&\le \norm{\lambda_M}\Er{Q}_{T,M}.&&(\because \text{Cauchy-Schwarz, \eqref{eq:Err1_g_T}})\notag
\end{align}
Since $\lambda_M$ solves \eqref{eq:D}, and hence minimizes $J_M$, we obtain
\begin{align*}
\lambda_M&=\argmin_{\lambda\in\R^L}J_M(\lambda)\subset \set{\lambda \mid J_M(\lambda)\le J_M(0)}\\
&\subset \set{\lambda \left| 1-\Er{Q}_{T,M}\norm{\lambda}+\frac{1}{2}C\norm{\lambda}^2\le 1\right.}. &&(\because \eqref{eq:J_quad}, J_M(0)=1)
\end{align*}
Solving the inequality, we obtain $\norm{\lambda_M}\le \frac{2\Er{Q}_{T,M}}{C}$.  \eqref{eq:H_estim} follows from \eqref{eq:H_ub}.
\end{proof}

Since $-\log(1-t)\approx t<2t$ for small $t>0$, the error bound \eqref{eq:final_estim} is tighter than \eqref{eq:alt_estim} for large $M$ by setting $t=\frac{(\Er{Q}_{T,M})^2}{2C}$.

Theorems \ref{thm:MainErrorEstimate} and \ref{thm:alt_error} show that the order of the theoretical error of our approximation method is at most that of the initial quadrature formula, but does not say that the error actually improves.  Next we show that our method improves the accuracy of the integration in some situation.

Assume that the density $f$ has a compact support.  Let $T(x)=(T_1(x),\dots,T_L(x))$ be the moment defining function, which are bounded on $\supp f$.  Consider approximating the integrand $g$ using the components of $T$ as basis functions:
$$g(x)\approx b_{g,T}(x)=\sum_{l=1}^L \beta_lT_l(x),$$
where $\beta_1,\dots,\beta_L$ are coefficients.  Let
\begin{equation}
r_{g,T}=\frac{g(x)-b_{g,T}(x)}{\norm{g-b_{g,T}}_{\infty}}\label{eq:g_res}
\end{equation}
be the normalized residual of the approximation.  Clearly $\norm{r_{g,T}}_\infty\le 1$.

\begin{thm}\label{thm:g_res}
Let everything be as in Theorem \ref{thm:alt_error} and $\supp f$ be compact.  Then
\begin{equation}
\Er{P}_{g,M}\le \norm{g-b_{g,T}}_\infty\left(\Er{Q}_{r_{g,T},M}+\frac{2}{\sqrt{C}}\Er{Q}_{T,M}\right).\label{eq:improved_error}
\end{equation}
\end{thm}

\begin{proof}
Since the moments $T_1,\dots,T_L$ are exact under $P_M=\set{p_{i,M}}$, we have $\Er{P}_{b_{g,T},M}=0$.  Therefore by the triangle inequality we obtain
\begin{align*}
\Er{P}_{g,M}&\le \Er{P}_{g-b_{g,T},M}+\Er{P}_{b_{g,T},M}=\Er{P}_{g-b_{g,T},M}\\
&\le \Er{Q}_{g-b_{g,T},M}+\norm{g-b_{g,T}}_\infty \frac{2}{\sqrt{C}}\Er{Q}_{T,M}&& (\because \text{Theorem \ref{thm:alt_error}})\\
&=\norm{g-b_{g,T}}_\infty\left(\Er{Q}_{r_{g,T},M}+\frac{2}{\sqrt{C}}\Er{Q}_{T,M}\right),&&(\because \eqref{eq:g_res})
\end{align*}
which is \eqref{eq:improved_error}.
\end{proof}

A similar bound can be obtained if we apply Theorem \ref{thm:MainErrorEstimate} instead of \ref{thm:alt_error}.  Theorem \ref{thm:g_res} shows that by matching the moments $T(x)$, the error improves by the factor $\norm{g-b_{g,T}}_\infty$, which is the approximation error of the integrand $g$ by a linear combination of the moments $T(x)$.  

For example, suppose that the problem is one dimensional with $\supp f=[c,d]$ and we match the polynomial moments by setting the moment defining function $T_{l}(x) = x^{l}\ (l = 1, \ldots, L)$.  Then for $b_{g,T}$ above we can adopt the Chebyshev interpolating polynomial of $g$, which is an $L$-degree polynomial that coincides with $g$ at the points $x_{j} = \frac{c+d}{2} + \frac{d-c}{2} \cos(j \pi/L)$, where $j = 0,1,\dots, L$.  The Chebyshev interpolating polynomial can be easily computed and is known to be a nearly optimal approximating polynomial of a continuous function on a finite interval \cite[Ch.~16]{trefethen2013}.\footnote{Although the theoretical error estimate of the Chebyshev interpolation is well-known \cite[Ch.~7]{trefethen2013}, we do not use the estimate in the next section because it is not so tight to explain the improvement by our method when $L$ is small.  Instead we use the actual computed values of the error $\norm{g - b_{g,T}}_{\infty}$ as an improvement factor.}

\section{Numerical experiments}
\label{sec:Num}

In this section, we present some numerical examples that compare the accuracy of the approximate expectations computed by an initial quadrature formula and its modifications by our proposed method.  
All computations in this section are done by MATLAB programs with double precision floating point arithmetic on a PC.

\subsection{Beta and uniform distributions}

For simplicity, we consider continuous probability distributions on the finite interval $[0,1]$, specifically the beta distributions and the uniform distribution.  The exact density function of the beta distribution is
$$f(x) = x^{a-1} (1-x)^{b-1}/B(a,b),$$
where $B(\, \cdot\, ,\, \cdot\, )$ is the beta function.  We set $(a,b) = (1,3)$ and $(2,4)$ in the following. 

\subsubsection{Initial discretization and implementation}
We adopt the trapezoidal formula and Simpson's formula as initial quadrature formulas.  Consider the discrete set $D_M = \set{ mh_{M} \mid m = 0, 1, \dots, 2M}$, where $h_M=\frac{1}{2M}$ is the distance between the points and $M=1,2,\dots,12$.  The number of points is $I_M=2M+1$ and the integration points are $x_{i,M} = (i-1) h_{M}$, where $i=1,\dots,I_M$.  The integration weights are
\begin{align*}
&\text{Trapezoidal:} & w_{i, M} &= 
\begin{cases}
h_{M}, & (i \neq 1, I_M)\\
h_{M}/2, & (i = 1, I_M) 
\end{cases}\\
&\text{Simpson:} & w_{i, M} &= 
\begin{cases}
4h_{M}/3, & (i \neq 1, I_M \text{ and } i \text{ is even} )\\
2h_{M}/3, & (i \neq 1, I_M \text{ and } i \text{ is odd} )\\
h_{M}/3. & (i = 1, I_M)
\end{cases}
\end{align*}

Below, we compute the approximate expectation $\E[g(X)]$ of a test function $g(x)$ using eight formulas: the trapezoidal formula, Simpson's formula, and their modifications by our proposed method with exact polynomial moments $\E[X^{l}]$ up to 2nd order $(l=1,2)$, 4th order $(l=1,\dots,4)$, and 6th order $(l=1,\dots, 6)$.  

We numerically solve the dual problem \eqref{eq:D} as follows.  First, note that in order for \eqref{eq:P} to have a solution, it is necessary that there are at least as many unknown variables ($p_{i,M}$'s, so in total $I_M$) as the number of constraints ($L$ moment constraints and $+1$ for probabilities to add up to 1, so $L+1$).  Thus we need $I_M\ge L+1$.\footnote{Since the beta density is zero at $x=0,1$, which are included in $x_{i,M}$'s, we necessarily have $p(x_{i,M})=0$ for $i=1,I_M$.  Thus, the number of unknown variables is $I_M-2$, so we need $I_M-2\ge L+1\iff I_M\ge L+3$ in the case of the beta distribution.}  A sufficient condition for the existence of a solution is $\bar{T}\in\co T(D)$ (Theorem \ref{thm:tanaka-toda}), which can be easily verified in the current application.

Second, note that \eqref{eq:D} is equivalent to the minimization of $J_M(\lambda)$ in \eqref{eq:Def_J}, which is a strictly convex function of $\lambda$.  In order to minimize $J_{M}$, we apply a variant of the Newton-Raphson algorithm.  Starting with $\lambda_0=0$, we iterate
\begin{equation}
\lambda_{n+1}=\lambda_n-[\kappa I+\nabla^2 J_{M}(\lambda_n)]^{-1}\nabla J_{M}(\lambda_n)\label{eq:newton}
\end{equation}
over $n=0,1,\dotsc$, where $\kappa>0$ is a small number, $I$ is the $L$-dimensional identity matrix, and $\nabla J_{M}, \nabla^2 J_{M}$ denote the gradient and the Hessian of $J_{M}$.  Such an algorithm is advocated in \cite{luenberger-ye2008}.  The Newton-Raphson algorithm corresponds to setting $\kappa=0$ in \eqref{eq:newton}.  Since the Hessian $\nabla^2 J_{M}$ is often nearly singular, the presence of $\kappa>0$ stabilizes the iteration \eqref{eq:newton}.  Below we set $\kappa=10^{-7}$ and terminate the iteration \eqref{eq:newton} when $\norm{\lambda_{n+1}-\lambda_n}<10^{-10}$.

\subsubsection{Results for the beta distributions}\label{subsec:Num.1}
For the test function, we pick $g(x) = \e^{x}$ for $x \in [0,1]$.  
The exact expectations are $\E[g(X)] = 3(-5 + 2 \mathrm{e})$ for 
$X \sim \Be(1,3)$ and $\E[g(X)] = 20(49-18 \mathrm{e})$ for $X \sim \Be(2,4)$.

Figure \ref{fig:Beta} shows the results.  According to the figures, our proposed method excels the trapezoidal and Simpson's formula in the accuracy.  The errors basically decrease as the order of the matched moments increases, consistent with Theorem \ref{thm:g_res}. 
Note that the curves marked ``Theoretical'' in Figure \ref{fig:Beta} 
express not exact error estimate but only the order of the theoretical error ($\mathrm{O}(M^{-2})$ for trapezoidal and $\mathrm{O}(M^{-4})$ for Simpson), which is same for Figures \ref{fig:UniformTrap} and \ref{fig:UniformSimp} below.

\begin{figure}[htbp]
\centering
\subfigure[$X \sim \Be(1,3)$, Trapezoidal.]{\includegraphics[width=0.45\linewidth]{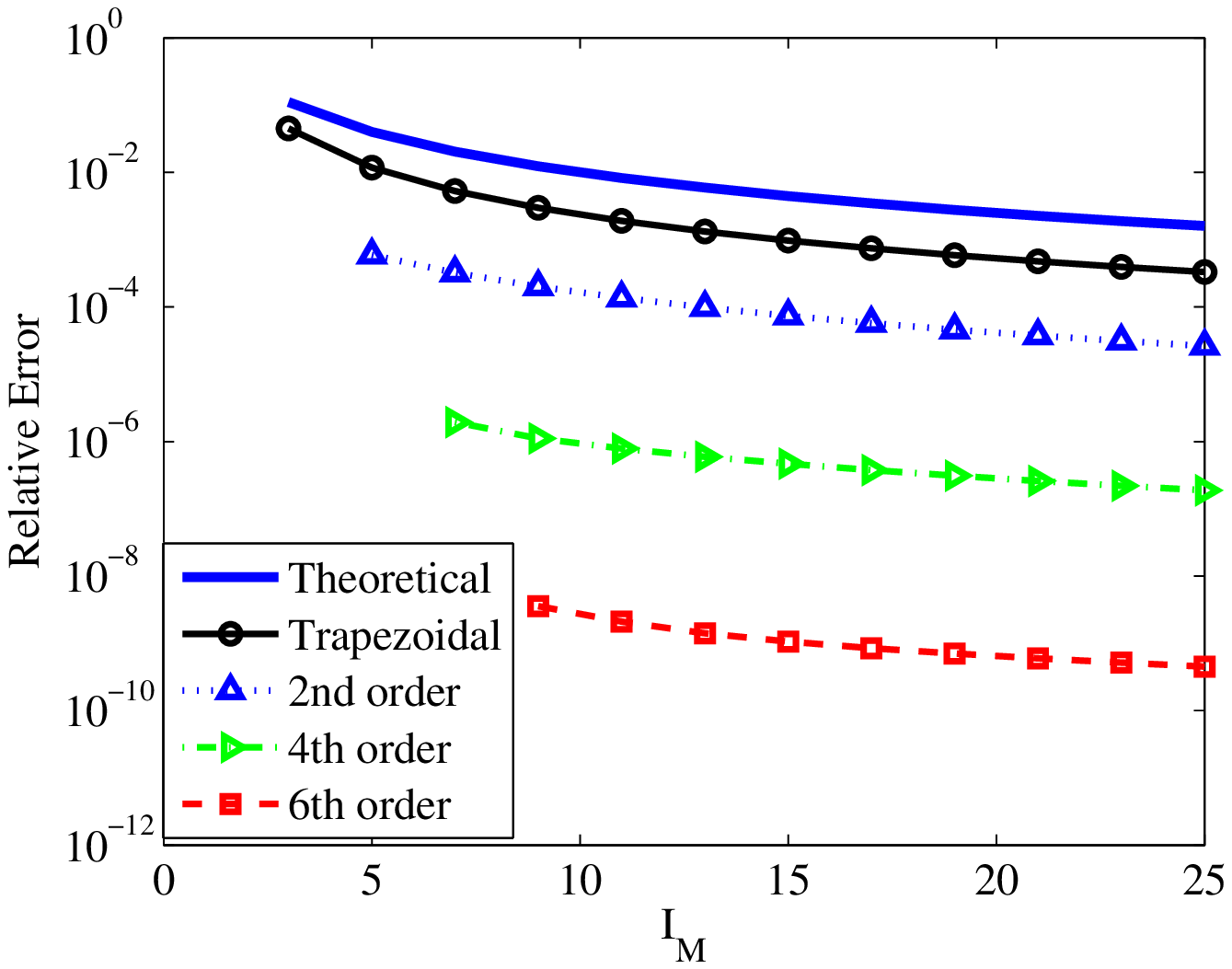}\label{fig:Beta13_T}}
\subfigure[$X \sim \Be(1,3)$, Simpson.]{\includegraphics[width=0.45\linewidth]{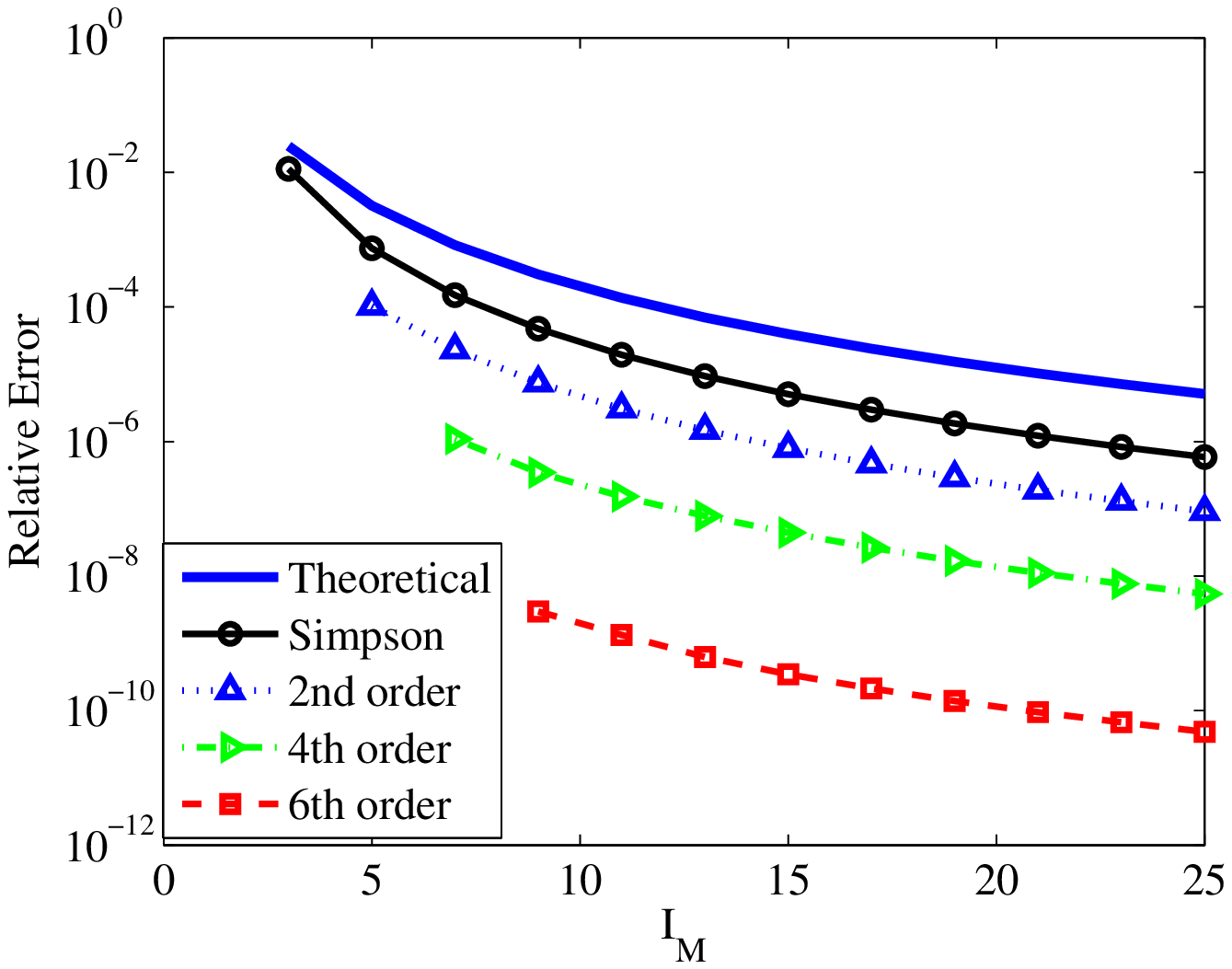}\label{fig:Beta13_S}}
\subfigure[$X \sim \Be(2,4)$, Trapezoidal.]{\includegraphics[width=0.45\linewidth]{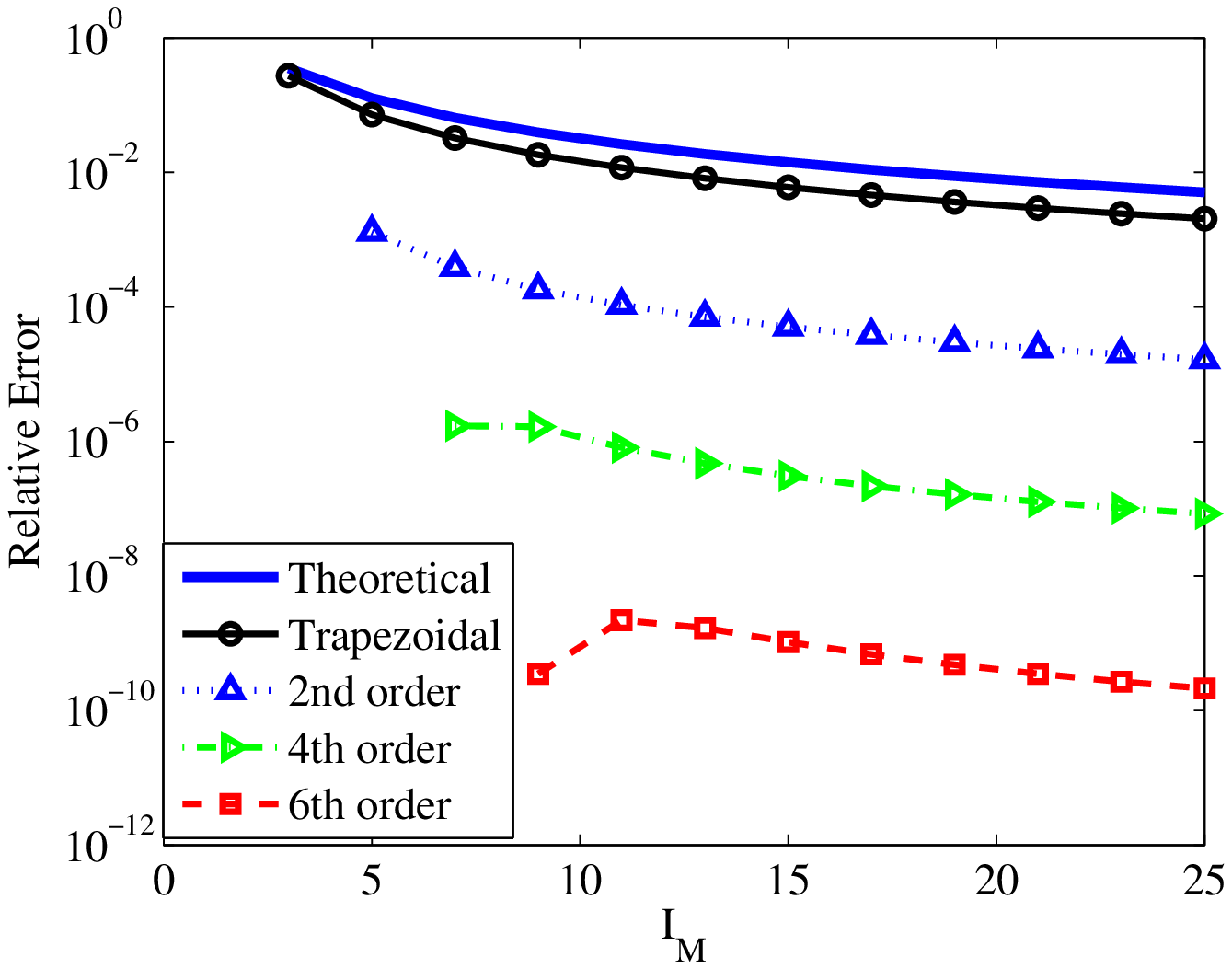}\label{fig:Beta24_T}}
\subfigure[$X \sim \Be(2,4)$, Simpson.]{\includegraphics[width=0.45\linewidth]{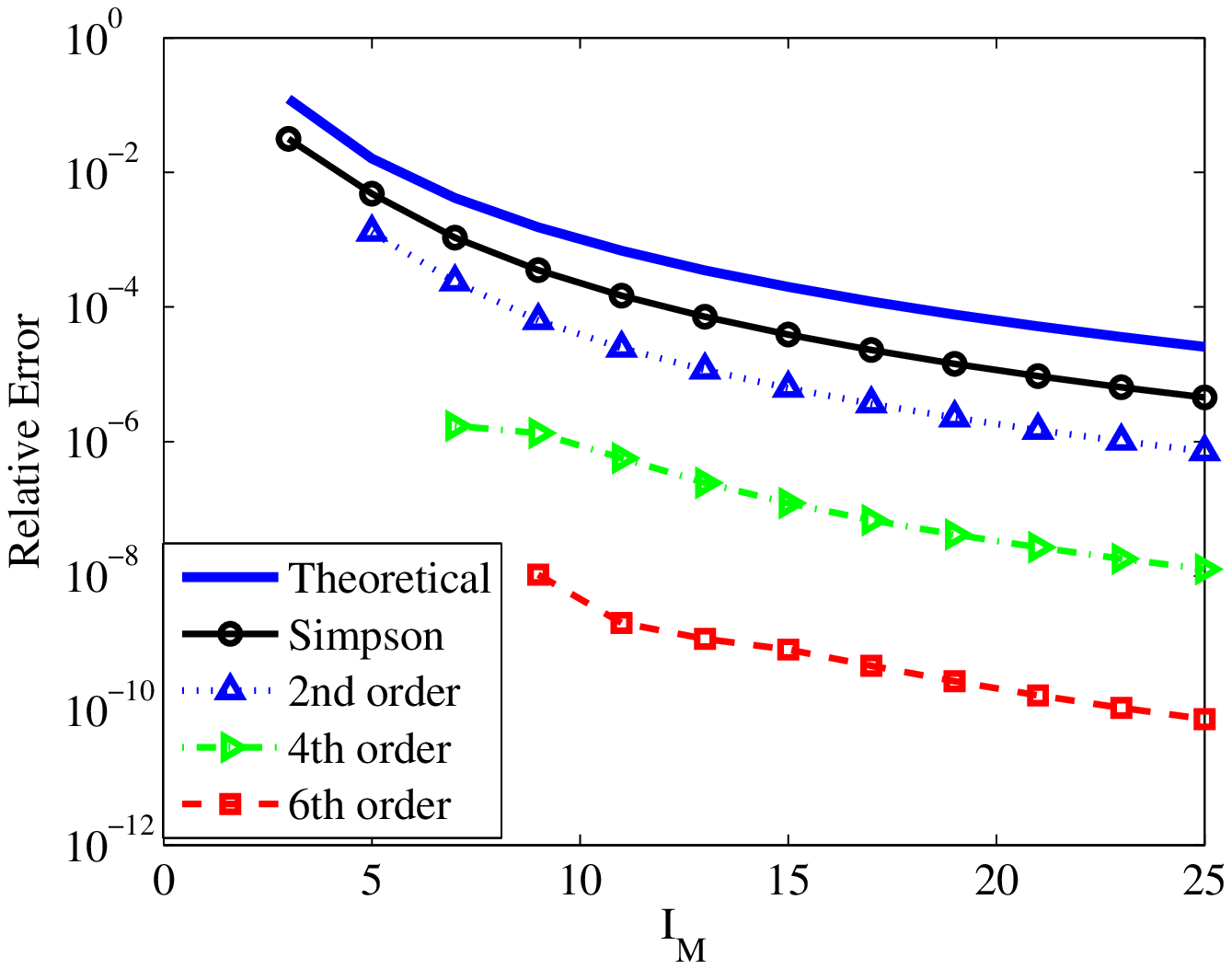}\label{fig:Beta24_S}}
\caption{Relative errors of the computed values of $\E[g(X)]$, where  $g(x)=\e^x$.  
The legends ``2nd order'', ``4th order'', and ``6th order'' represent those by our method with exact polynomial moments $\E[X^{l}]$ up to 2nd order $(l=1,2)$, 4th order $(l=1,\ldots, 4)$, and 6th order $(l=1,\ldots, 6)$, respectively.}\label{fig:Beta}
\subfigure[$\Be(1,3)$.]{\includegraphics[width=0.45\linewidth]{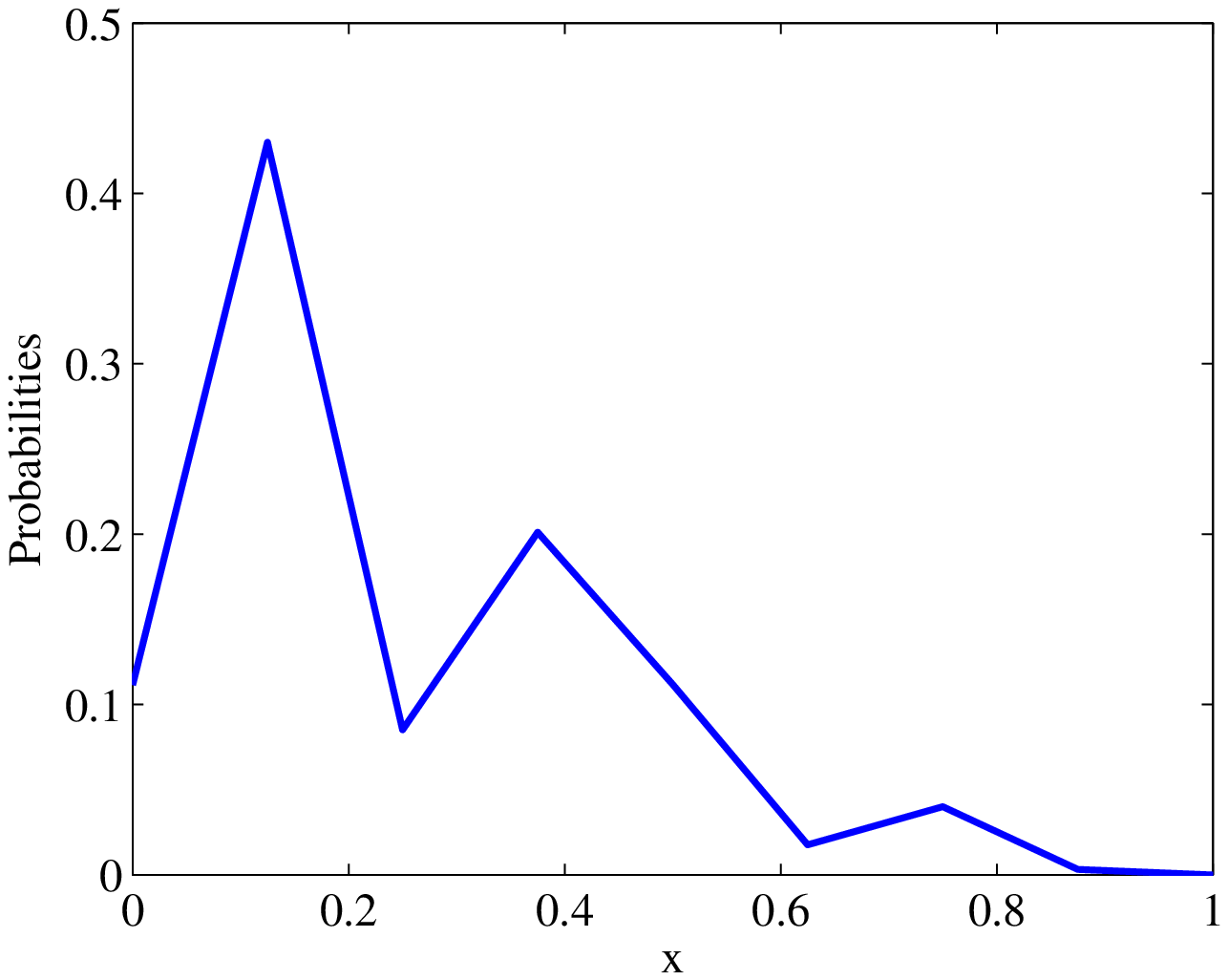}\label{fig:Beta13_Pathological}}
\subfigure[$\Be(2,4)$.]{\includegraphics[width=0.45\linewidth]{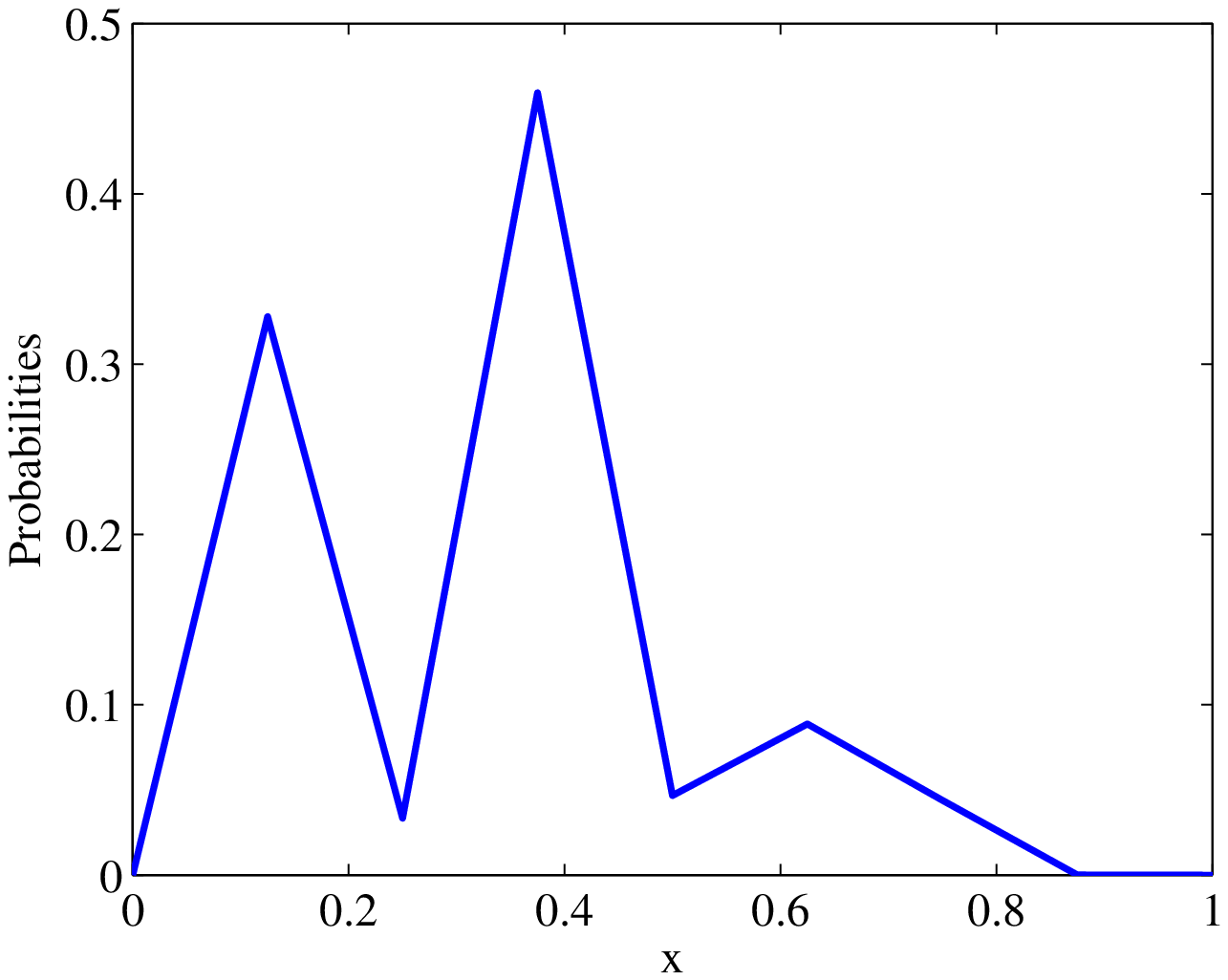}\label{fig:Beta24_Pathological}}
\caption{6th order discrete approximation from the trapezoidal formula with $I_M=9$ grid points.}\label{fig:pathological}
\end{figure}

The reason why the error curves for our proposed method are not necessarily parallel to those for the initial quadrature formula when $M$ is very small is because when the number of constraints $L+1$ is large relative to the number of unknown variables $I_M$, the method generates pathological probability distributions at the expense of matching many moments.  For instance, Figure \ref{fig:pathological} shows the discrete approximations of the beta distributions $\Be(1,3)$ and $\Be(2,4)$ with $L=6$ exact polynomial moments and $I_M=9$ grid points.  Clearly the discrete approximations do not resemble the continuous counterparts.  This pathological behavior is rarely an issue, however.  As long as there are twice as many grid points as constraints ($I_M\ge 2(L+1)$), the discrete approximation is well-behaved.

\subsubsection{Results for the uniform distribution}
We choose the test functions $g_1(x)=x^\frac{9}{2}$, $g_2(x)=\frac{1}{1+x}$, $g_3(x)=\sin(\pi x)$, and $g_4(x)=\log(1+x)$ for the uniform distribution on $[0,1]$.

Figures \ref{fig:UniformTrap} and \ref{fig:UniformSimp} show the numerical results.  In all cases, our method excels the initial quadrature formula.  
The improvement in the accuracy is significant (of the order $10^{-4}$ for the trapezoidal formula and $10^{-2}$ for Simpson's formula), consistent with Theorem \ref{thm:g_res}.  

\begin{figure}[htbp]
\centering
\subfigure[$g_1(x)=x^\frac{9}{2}$.]{\includegraphics[width=0.45\linewidth]{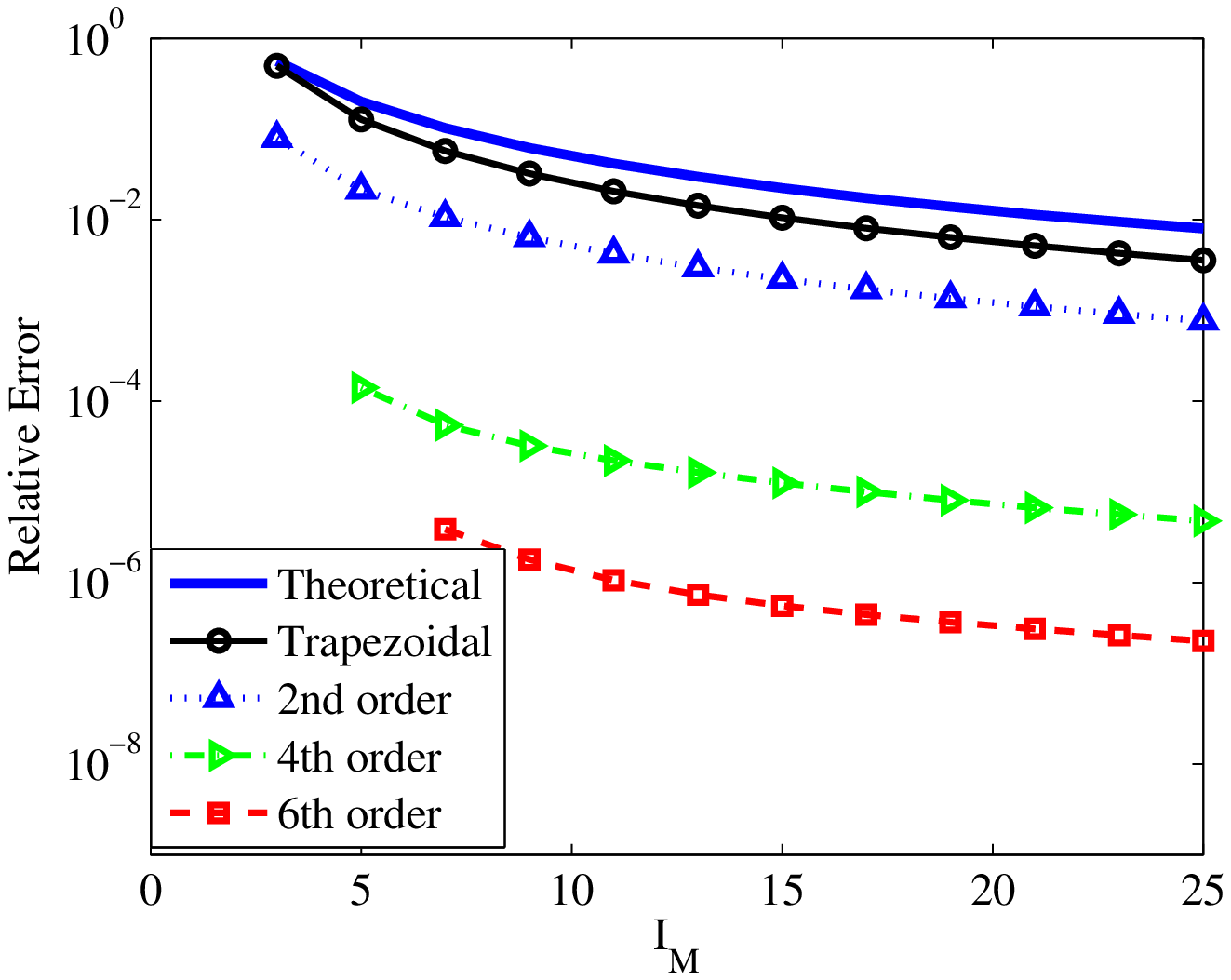}\label{fig:Err_Eg_Uni_Trap_x5ov2}}
\subfigure[$g_2(x)=\frac{1}{1+x}$.]{\includegraphics[width=0.45\linewidth]{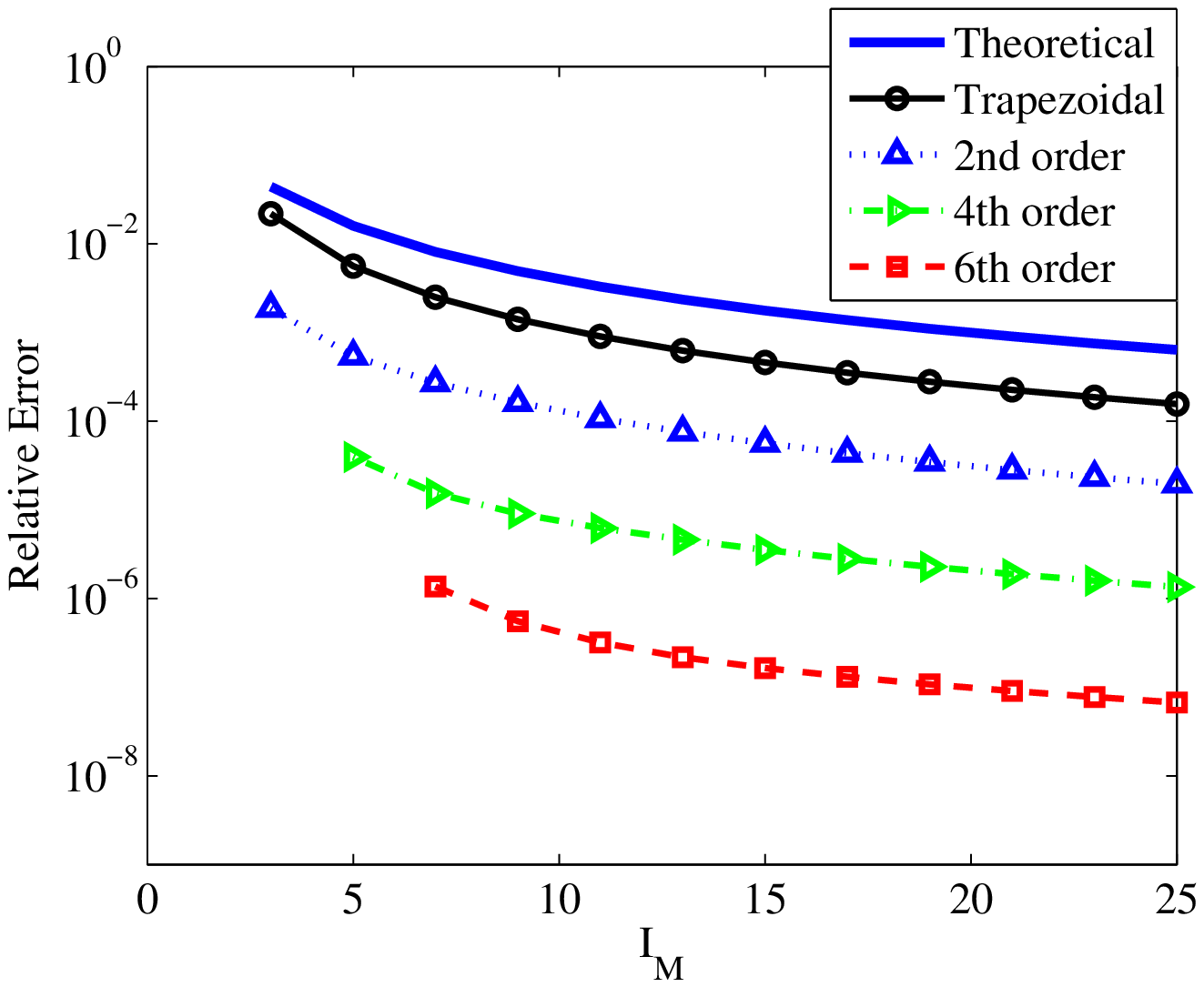}\label{fig:Err_Eg_Uni_Trap_Hyp}}
\subfigure[$g_3(x)=\sin(\pi x)$.]{\includegraphics[width=0.45\linewidth]{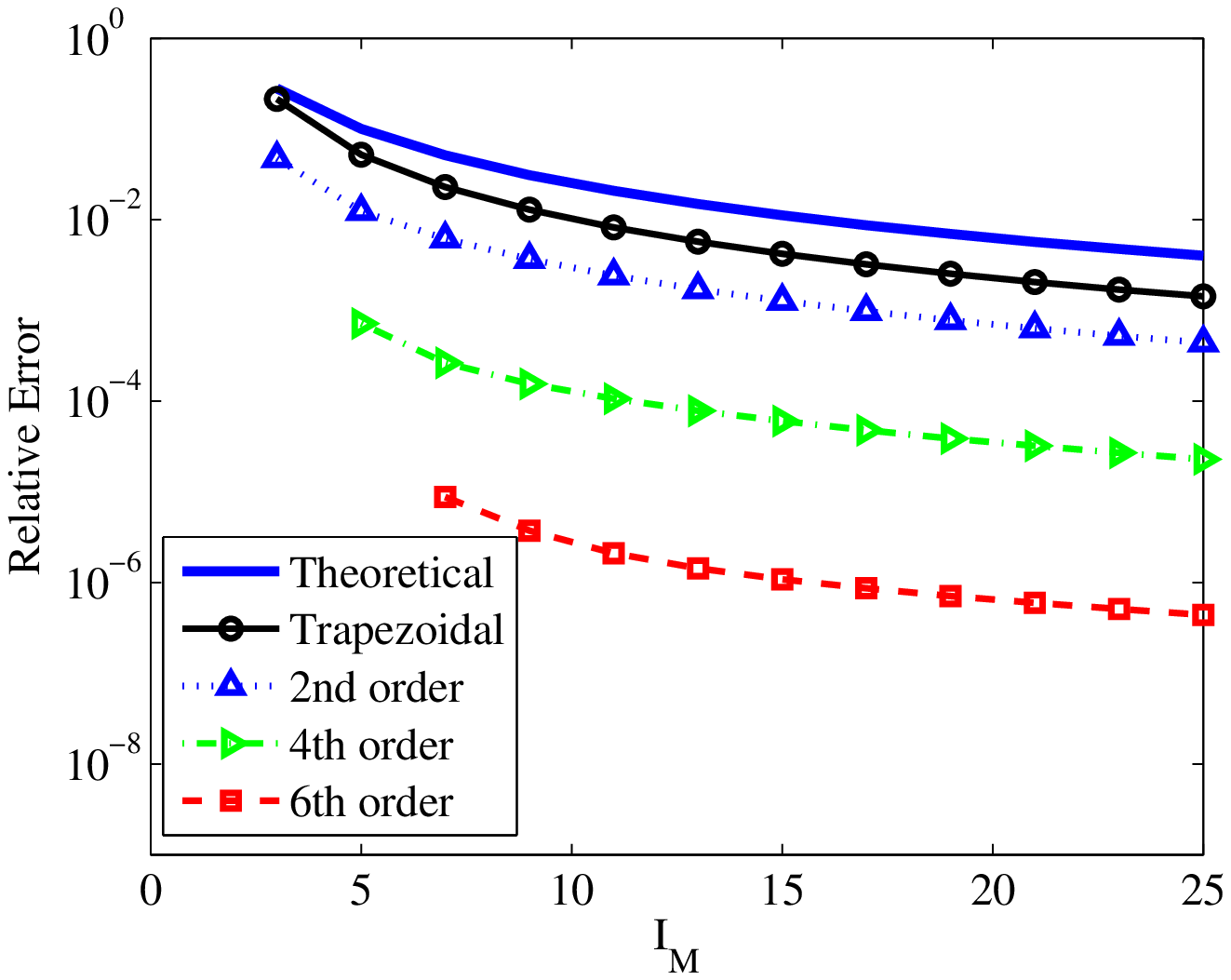}\label{fig:Err_Eg_Uni_Trap_Sin}}
\subfigure[$g_4(x)=\log(1+x)$.]{\includegraphics[width=0.45\linewidth]{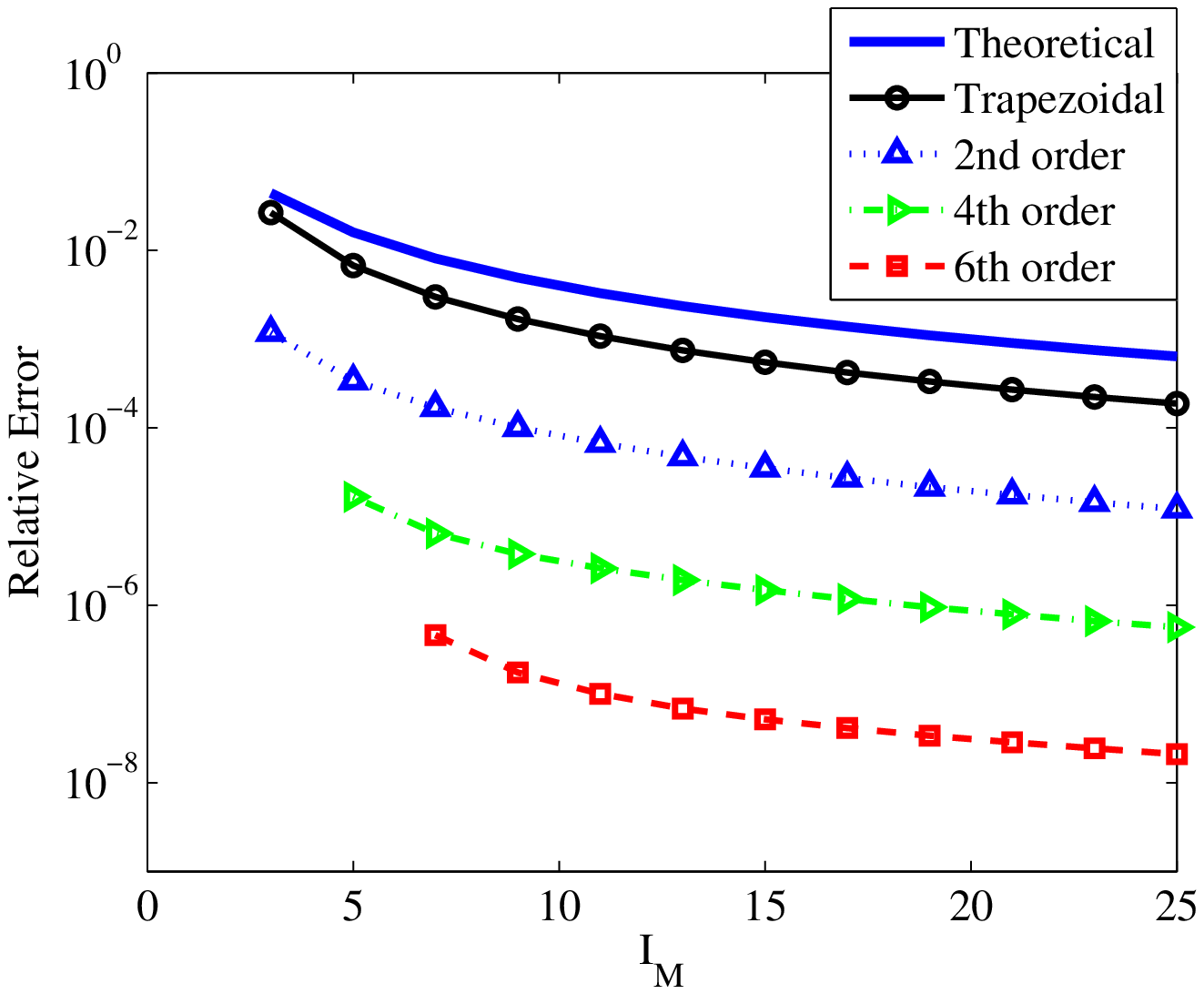}\label{fig:Err_Eg_Uni_Trap_Log}}
\caption{Relative errors of the computed values of $\int_0^1g(x)\dd x$.  The legends ``2nd order'', ``4th order'', and ``6th order'' represent those by our method with exact polynomial moments $\E[X^{l}]$ up to 2nd order $(l=1,2)$, 4th order $(l=1,\ldots, 4)$, and 6th order $(l=1,\ldots, 6)$, respectively.}\label{fig:UniformTrap}
\end{figure}

\begin{figure}[htbp]
\centering
\subfigure[$g_1(x)=x^\frac{9}{2}$.]{\includegraphics[width=0.45\linewidth]{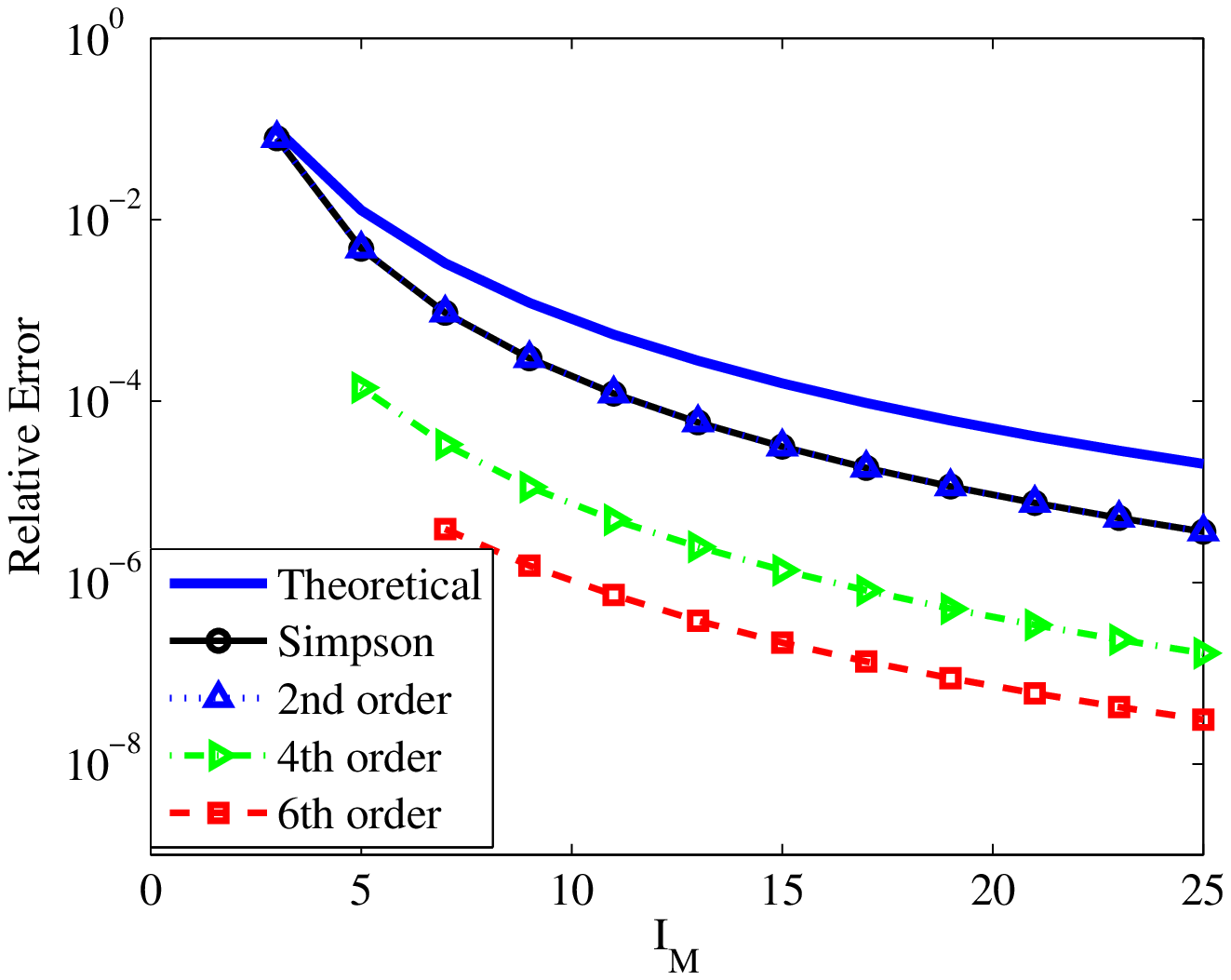}\label{fig:Err_Eg_Uni_Simp_x5ov2}}
\subfigure[$g_2(x)=\frac{1}{1+x}$.]{\includegraphics[width=0.45\linewidth]{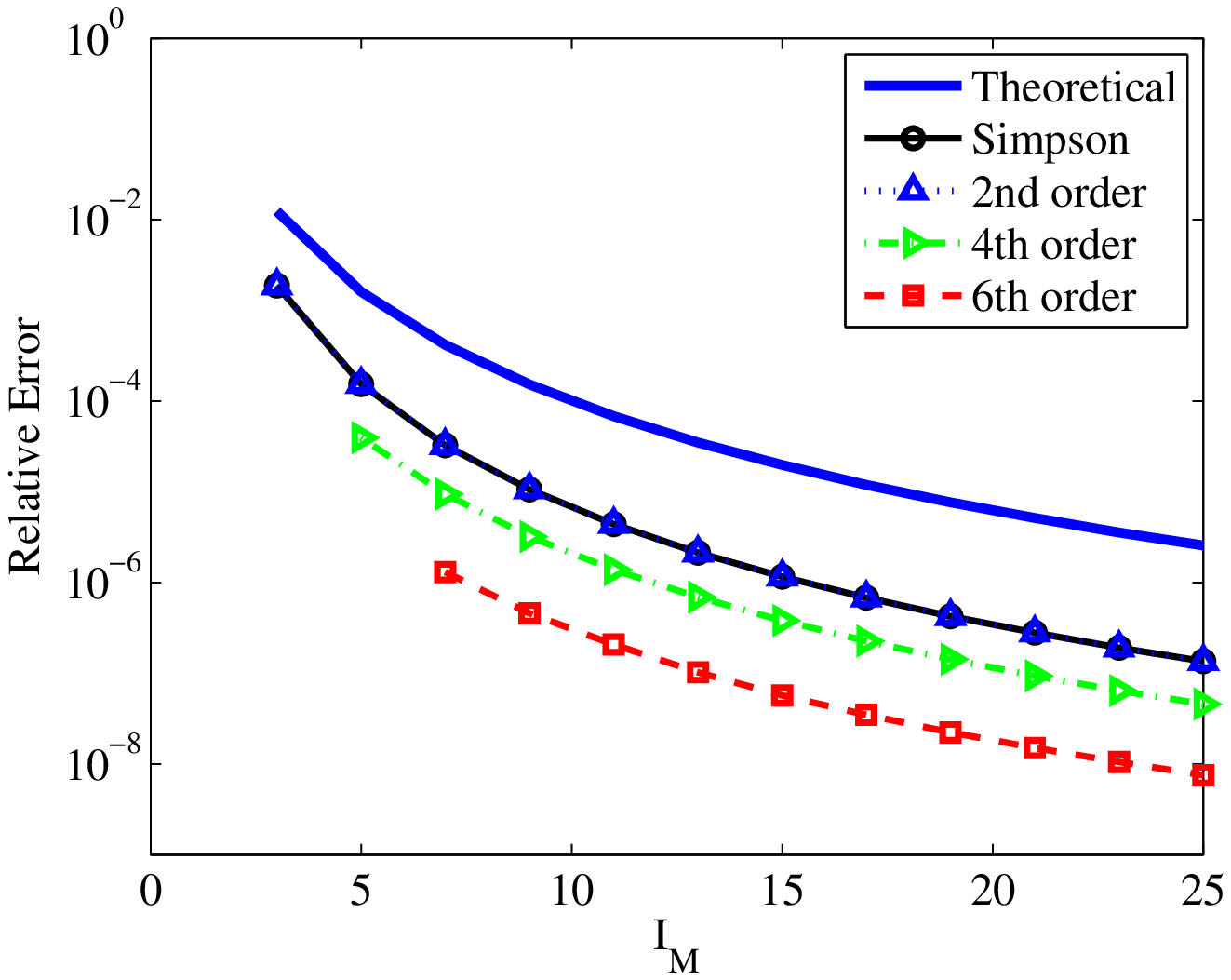}\label{fig:Err_Eg_Uni_Simp_Hyp}}
\subfigure[$g_3(x)=\sin(\pi x)$.]{\includegraphics[width=0.45\linewidth]{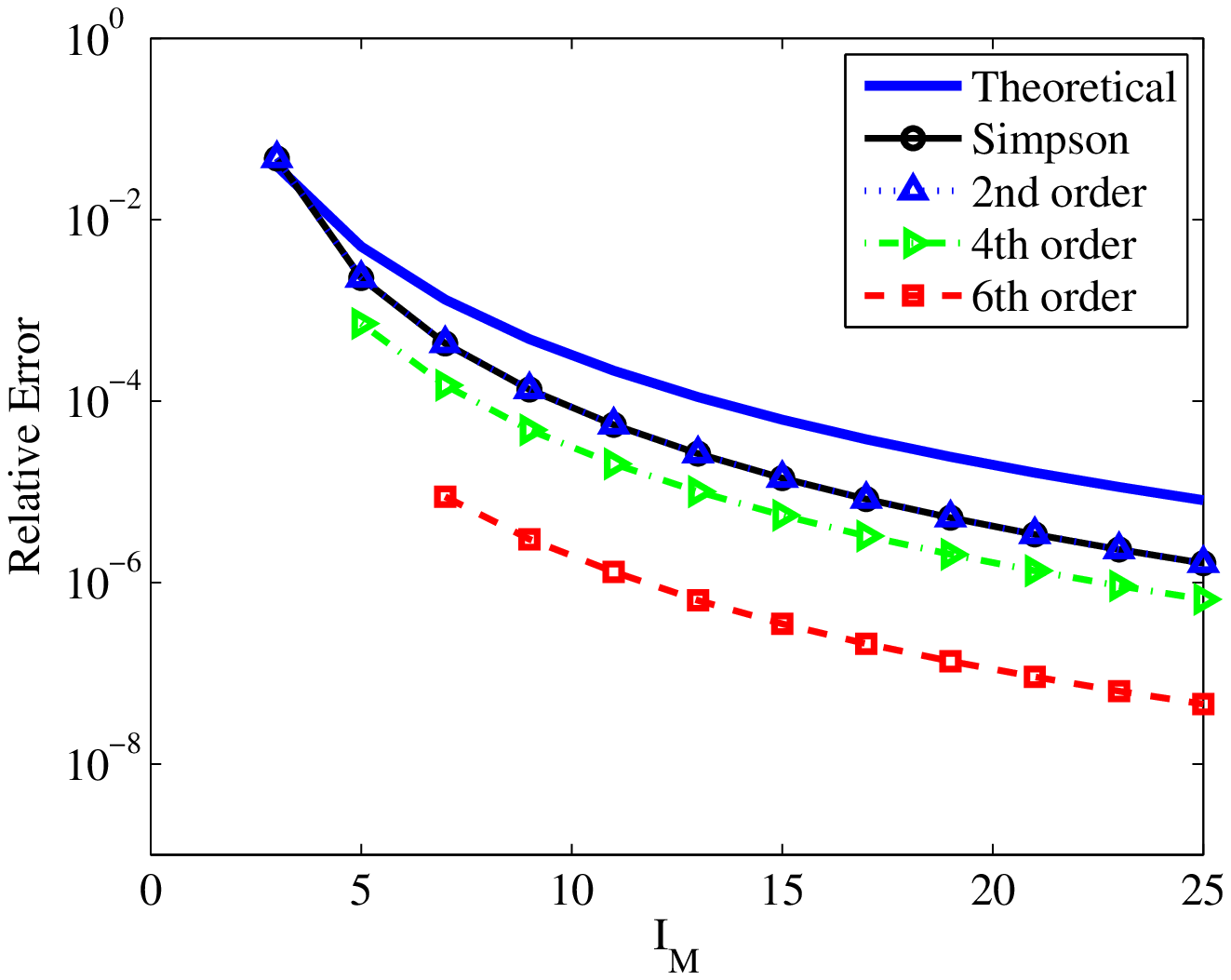}\label{fig:Err_Eg_Uni_Simp_Sin}}
\subfigure[$g_4(x)=\log(1+x)$.]{\includegraphics[width=0.45\linewidth]{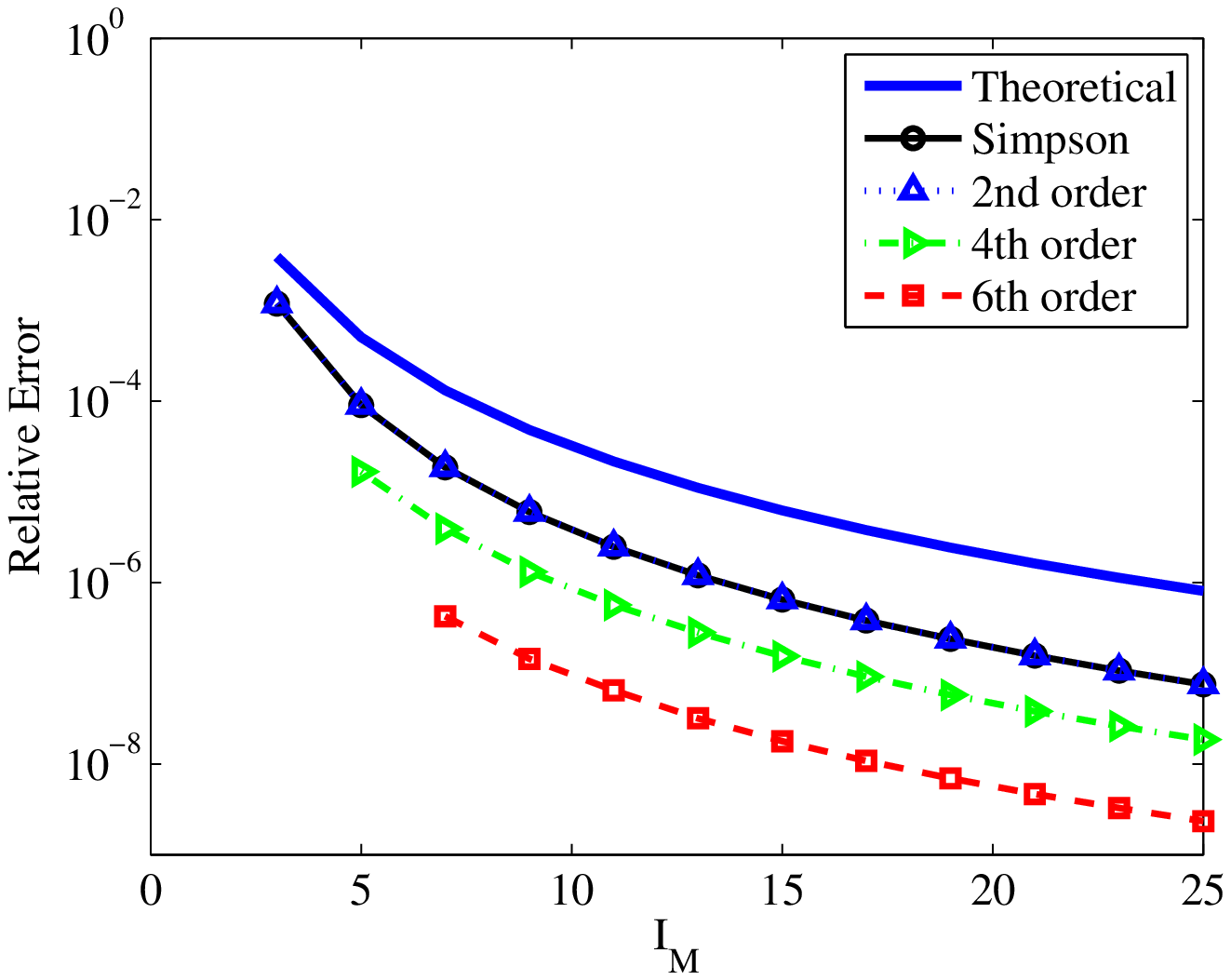}\label{fig:Err_Eg_Uni_Simp_Log}}
\caption{Relative errors of the computed values of $\int_0^1g(x)\dd x$.  The legends ``2nd order'', ``4th order'', and ``6th order'' represent those by our method with exact polynomial moments $\E[X^{l}]$ up to 2nd order $(l=1,2)$, 4th order $(l=1,\ldots, 4)$, and 6th order $(l=1,\ldots, 6)$, respectively.}\label{fig:UniformSimp}
\end{figure}

\subsubsection{Discussion}

According to Figures \ref{fig:Beta}, \ref{fig:UniformTrap}, and \ref{fig:UniformSimp}, the rate of convergence (the slope of the relative error with respect to the number of grid points $I_{M}$) is almost the same for the initial quadrature formula and our method.  Specifically, the curves of the relative error are quite similar to the translations of the graphs of the theoretical errors, which are $\mathrm{O}(M^{-2})$ for the trapezoidal formula and  $\mathrm{O}(M^{-4})$ for Simpson's formula when the integrands are in $C^{2}[0,1]$ and $C^{4}[0,1]$, respectively \cite[Ch.~2]{davis-rabinowitz1984}.
This observation is consistent with Theorem \ref{thm:g_res}, which shows that the error estimate is of the same order as the quadrature formula but improves by the error of the Chebyshev approximation.

Table \ref{tab:ErrCheb} shows the errors of the Chebyshev approximations of the functions $g, g_{1}, \dots, g_{4}$.  Since the trapezoidal formula does not even match the second moment, the theoretical improvement of our method (in $\log_{10}$) should be the numbers in Table \ref{tab:ErrCheb}.  The actual improvements in Figures \ref{fig:Beta} and \ref{fig:UniformTrap} are in line with these numbers.  Since Simpson's formula gives exact quadratures for 2nd order polynomials (hence our method with $L=2$ coincides with Simpson), the theoretical improvement of our method should be the \emph{difference} between the numbers in Table \ref{tab:ErrCheb} and those in the row corresponding to $L=2$.  For example, when we match $L=6$ moments, the improvement should be approximately $10^{-3}$, which is similar to what we observe in Figures \ref{fig:Beta} and \ref{fig:UniformSimp}. 


\begin{table}[htbp]
\begin{center}
\caption{Errors of the Chebyshev approximations of $g, g_{1}, \ldots, g_{4}$ in $\log_{10}$.}
\label{tab:ErrCheb}
\begin{tabular}{c | ccccc }
& \multicolumn{5}{c}{Test function}\\
Degree & $\e^{x}$ & $x^\frac{9}{2}$ & $(1+x)^{-1}$ & $\sin(\pi x)$ & $\log(1+x)$ \\
\hline
2 &  $-1.841$ &  $-0.847$  &  $-1.874$  &  $-1.251$ &  $-2.221$ \\
4 &  $-4.285$ &  $-2.869$  &  $-3.363$  &  $-3.031$ &  $-3.918$ \\
6 &  $-7.102$ &  $-5.048$  &  $-4.895$  &  $-4.872$ &  $-5.592$  
\end{tabular}
\end{center}
\end{table}

In summary, our method seems to be particularly suited for fine-tuning a quadrature formula with a small number of integration points. For instance, we can construct a highly accurate compound rule by subdividing the interval and applying our method to each subinterval.

\subsection{Optimal portfolio problem}\label{subsec:Num.2}
In this section we numerically solve the optimal portfolio problem briefly discussed in the introduction (see \cite{tanaka-toda-maxent-approx} for more details).  Suppose that there are two assets, stock and bond, with gross returns $R_1,R_2$.  Asset 1 (stock) is stochastic and lognormally distributed: $\log R_1\sim N(\mu,\sigma^2)$, where $\mu$ is the expected return and $\sigma$ is the volatility.  Asset 2 (bond) is risk-free and $\log R_2=r$, where $r$ is the (continuously compounded) interest rate.  
The optimal portfolio $\theta$ is determined by the optimization
\begin{equation}
U = \max_{\theta}\frac{1}{1-\gamma}\E[(R_1\theta+R_2(1-\theta))^{1-\gamma}],\label{optimalportfolio}
\end{equation}
where $\gamma>0$ is the relative risk aversion coefficient.  We set the parameters such that $\gamma=3$, $\mu=0.07$, $\sigma=0.2$, and $r=0.01$.  We numerically solve the optimal portfolio problem \eqref{optimalportfolio} in two ways, applying the trapezoidal formula and our proposed method.  (We also tried Simpson's method but it was similar to the trapezoidal method.)  To approximate the lognormal distribution, let $I_M=2M+1$ be the number of grid points ($M$ is the number of positive grid points) and $D_M=\set{mh\mid m=0,\pm 1,\dots,\pm M}$, where $h=1/\sqrt{M}$ is the grid size.  Let $p(x)$ be the approximating discrete distribution of $N(0,1)$ as before (trapezoidal or the proposed method with various moments).  Then we put probability $p(x)$ on the point $\e^{\mu+\sigma x}$ for each $x\in D_M$ to obtain the approximate stock return $R_1$.

Table \ref{t:optinvest} shows the optimal portfolio $\theta$ and its relative error for various moments $L$ and number of points $I_M=2M+1$.  The result is somewhat surprising.  Even with 3 approximating points ($M=1$), our proposed method derives an optimal portfolio that is off by only 0.5\% to the true value, whereas the trapezoidal method is off by 127\%.  While the proposed method virtually obtains the true value with 9 points ($M=4$, especially when the 4th moment is matched), the trapezoidal method still has 23\% of error.

\begin{table}[htbp]
\centering
\caption{Optimal portfolio and relative error for the trapezoidal method and our method.  $M$: number of positive grid points, $I_M=2M+1$: total number of grid points, $L$: maximum order of moments.}\label{t:optinvest}
\begin{tabular}{|cc|cc|cc|cc|}
\hline
\multicolumn{2}{|c|}{\# of grid points} & \multicolumn{2}{c|}{$L=0$ (trapezoidal)} & \multicolumn{2}{c|}{$L=2$} & \multicolumn{2}{c|}{$L=4$}\\ \hline
$M$ & $I_M$ & $\theta$ & Error (\%) & $\theta$ & Error (\%) & $\theta$ & Error (\%)\\ \hline
$1^2=1$ & 3 & 1.5155  & 127  & 0.6717 & 0.54  & - & - \\
$2^2=4$ & 9 & 0.8246  & 23.4  & 0.6694 & 0.20  & 0.6680 & $-0.015$ \\
$3^2=9$ & 19 & 0.6830 & 2.24  & 0.6684 & 0.044 & 0.6681 & 0 \\
$4^2=16$ & 33 & 0.6687 & 0.088 & 0.6682 & 0.015     & 0.6681 & 0 \\
$5^2=25$ & 51 & 0.6681 & 0		& 0.6681 & 0     & 0.6681 & 0 \\ \hline
\end{tabular}
\end{table}

The reason why the trapezoidal method gives poor results when the number of approximating points is small is because the moments are not matched.  To see this, taking the first-order condition for the optimal portfolio problem \eqref{optimalportfolio}, we obtain $\E[(\theta X+R_2)^{-\gamma}X]=0$, where $X=R_1-R_2$ is the excess return on the stock.  Using the Taylor approximation $x^{-\gamma}\approx a^{-\gamma}-\gamma a^{-\gamma-1}(x-a)$ for $x=\theta X+R_2$ and $a=\E[\theta X+R_2]$ and solving for $\theta$, after some algebra we get
$$\theta=\frac{R_2\E[X]}{\gamma\var[X]-\E[X]^2}.$$
Therefore the (approximate) optimal portfolio depends on the first and second moments of the excess return $X$.  Our method is accurate precisely because we match the moments.  In complex economic problems, oftentimes we cannot afford to use many integration points, in which case our method might be useful to obtain accurate results.

\section*{Acknowledgments}
We thank Jonathan Borwein, seminar participants at the 33rd International Workshop on Bayesian Inference and Maximum Entropy Methods in Science and Engineering (MaxEnt 2013) and 2014 Info-Metrics Institute Conference, and two anonymous referees for comments and feedback that greatly improved the paper.  KT was partially supported by JSPS KAKENHI Grant Number 24760064.

\appendix
\section{Proof of Pinsker's inequality}
This appendix proves Pinsker's inequality $\frac{1}{2}\norm{P-Q}_1^2\le H(P;Q)$, where $P=\set{p_n}_{n=1}^N$ and $Q=\set{q_n}_{n=1}^N$ are probability distributions and $\norm{\cdot}_1$ denotes the $L^1$ norm.  Let $N_+=\set{n \mid p_n>q_n}$, $N_-=\set{n \mid p_n\le q_n}$, $p=\sum_{n\in N_+}p_n$, and $q=\sum_{n\in N_+}q_n$.  Then
\begin{align*}
\norm{P-Q}_1=\sum_{n=1}^N\abs{p_n-q_n}&=\sum_{n\in N_+}(p_n-q_n)-\sum_{n\in N_-}(p_n-q_n)\\
&=2\sum_{n\in N_+}(p_n-q_n)-\sum_{n=1}^N(p_n-q_n)=2(p-q),
\end{align*}
where we have used $\sum_{n=1}^Np_n=\sum_{n=1}^Nq_n=1$.  By the convexity of $-\log(\cdot)$, we get
\begin{align*}
\sum_{n=1}^Np_n\log \frac{p_n}{q_n}&=p\sum_{n\in N_+}\frac{p_n}{p}\left(-\log\frac{q_n/p}{p_n/p}\right)+(1-p)\sum_{n\in N_-}\frac{p_n}{1-p}\left(-\log\frac{q_n/(1-p)}{p_n/(1-p)}\right)\\
&\ge -p\log\left(\sum_{n\in N_+}\frac{q_n}{p}\right)-(1-p)\log\left(\sum_{n\in N_-}\frac{q_n}{1-p}\right)\\
&=p\log \frac{p}{q}+(1-p)\log \frac{1-p}{1-q}.
\end{align*}
Therefore it suffices to show
$$h(p,q):=p\log \frac{p}{q}+(1-p)\log \frac{1-p}{1-q}-2(p-q)^2\ge 0.$$
Fix $q$ and regard the left-hand side as a function of $p$ alone.  Then
\begin{align*}
h'(p,q)&=\log \frac{p}{q}-\log\frac{1-p}{1-q}-4(p-q),\\
h''(p,q)&=\frac{1}{p}+\frac{1}{1-p}-4=\frac{(1-2p)^2}{p(1-p)}\ge 0,
\end{align*}
so $h$ is convex.  Clearly $h'(q,q)=0$, so it follows that $h(p,q)\ge h(q,q)=0$.

\bibliographystyle{plain}
\bibliography{reference}

\begin{thebibliography}{10}

\bibitem{abramov2007}
Rafail~V. Abramov.
\newblock An improved algorithm for the multidimensionalmoment-constrained
  maximum entropy problem.
\newblock {\em Journal of Computational Physics}, 226(1):621--644, September
  2007.

\bibitem{abramov2010}
Rafail~V. Abramov.
\newblock The multidimensional maximum entropy moment problem: A review of
  numerical methods.
\newblock {\em Communications in Mathematical Sciences}, 8(2):377--392, 2010.

\bibitem{adda-cooper2003}
J\'er\^ome Adda and Russel~W. Cooper.
\newblock {\em Dynamic Economics: Quantitative Methods and Applications}.
\newblock MIT Press, Cambridge, MA, 2003.

\bibitem{aiyagari1994}
S.~Rao Aiyagari.
\newblock Uninsured idiosyncratic risk and aggregate saving.
\newblock {\em Quarterly Journal of Economics}, 109(3):659--684, 1994.

\bibitem{AHOT2014}
Graham~W. Alldredge, Cory~D. Hauck, Dianne~P. O'{L}eary, and Andr{\'e}~L. Tits.
\newblock Adaptive change of basis in entropy-based moment closures for linear
  kinetic equations.
\newblock {\em Journal of Computational Physics}, 258(1):489--508, February
  2014.

\bibitem{barron-sheu1991}
Andrew~R. Barron and Chyong-Hwa Sheu.
\newblock Approximation of density functions by sequences of exponential
  families.
\newblock {\em Annals of Statistics}, 19(3):1347--1369, 1991.

\bibitem{borwein-lewis1991}
Jonathan~M. Borwein and Adrian~S. Lewis.
\newblock Duality relationships for entropy-like minimization problems.
\newblock {\em SIAM Journal on Control and Optimization}, 29(2):325--338, March
  1991.

\bibitem{borwein-lewis2006}
Jonathan~M. Borwein and Adrian~S. Lewis.
\newblock {\em Convex Analysis and Nonlinear Optimization: Theory and
  Examples}.
\newblock Canadian Mathematical Society Books in Mathematics. Springer, New
  York, 2nd edition, 2006.

\bibitem{buchen-kelly1996}
Peter~W. Buchen and Michael Kelly.
\newblock The maximum entropy distribution of an asset inferred from option
  prices.
\newblock {\em Journal of Financial and Quantitative Analysis}, 31(1):143--159,
  1996.

\bibitem{budisic-putinar2012}
Marko Budi\v{s}i{\'c} and Mihai Putinar.
\newblock Conditioning moments of singular measures for entropy optimization.
  {I}.
\newblock {\em Indagationes Mathematicae}, 23(4):848--883, December 2012.

\bibitem{caticha-giffin2006-AIP}
Ariel Caticha and Adom Giffin.
\newblock Updating probabilities.
\newblock In Ali Mohammad-Djafari, editor, {\em {B}ayesian Inference and
  Maximum Entropy Methods in Science and Engineering}, volume 872 of {\em AIP
  Conference Proceedings}, pages 31--42, 2006.

\bibitem{csiszar1967}
Imre Csisz{\'a}r.
\newblock Information-type measures of difference of probability distributions
  and indirect observations.
\newblock {\em Studia Scientiarum Mathematicarum Hungarica}, 2:299--318, 1967.

\bibitem{csiszar1984}
Imre Csisz{\'a}r.
\newblock {S}anov property, generalized {$I$}-projection and a conditional
  limit theorem.
\newblock {\em Annals of Probability}, 12(3):768--793, August 1984.

\bibitem{davis-rabinowitz1984}
Philip~J. Davis and Philip Rabinowitz.
\newblock {\em Methods of Numerical Integration}.
\newblock Academic Press, Orlando, FL, second edition, 1984.

\bibitem{DHLN1992}
Andr{\'e}e Decarreau, Danielle Hilhorst, Claude Lemar{\'e}chal, and Jorge
  Navaza.
\newblock Dual methods in entropy maximization. {A}pplication to some problems
  in crystallography.
\newblock {\em SIAM Journal on Optimization}, 2(2):173--197, 1992.

\bibitem{devuyst-preckel2007}
Eric~A. DeVuyst and Paul~V. Preckel.
\newblock {G}aussian cubature: A practitioner's guide.
\newblock {\em Mathematical and Computer Modelling}, 45(7-8):787--794, April
  2007.

\bibitem{fedotov-harremoes2003}
Alexei~A. Fedotov and Peter Harremo\"es.
\newblock Refinements of {P}insker's inequality.
\newblock {\em IEEE Transactions on Information Theory}, 49(6):1491--1498, June
  2003.

\bibitem{foley1994}
Duncan~K. Foley.
\newblock A statistical equilibrium theory of markets.
\newblock {\em Journal of Economic Theory}, 62(2):321--345, April 1994.

\bibitem{gospodinov-lkhagvasuren2014}
Nikolay Gospodinov and Damba Lkhagvasuren.
\newblock A moment-matching method for approximating vector autoregressive
  processes by finite-state {M}arkov chains.
\newblock {\em Journal of Applied Econometrics}, 29(5):843--859, August 2014.

\bibitem{hauck-levermore-tits2008}
Cory~D. Hauck, C.~David Levermore, and Andr{\'e}~L. Tits.
\newblock Convex duality and entropy-based moment closures: Characterizing
  degenerate densities.
\newblock {\em SIAM Journal on Control and Optimization}, 47(4):1977--2015,
  2008.

\bibitem{hoyland-wallace2001}
Kjetil H{\o}yland and Stein~W. Wallace.
\newblock Generating scenario trees for multistage decision problems.
\newblock {\em Management Science}, 47(2):295--307, February 2001.

\bibitem{huggett1993}
Mark Huggett.
\newblock The risk-free rate in heterogeneous-agent incomplete-insurance
  economies.
\newblock {\em Journal of Economic Dynamics and Control}, 17(5-6):953--969,
  September-November 1993.

\bibitem{jaynes1957a}
Edwin~T. Jaynes.
\newblock Information theory and statistical mechanics.
\newblock {\em Physical Review}, 106(4):620--630, May 1957.

\bibitem{jaynes1982}
Edwin~T. Jaynes.
\newblock On the rationale of maximum-entropy methods.
\newblock {\em Proceedings of the IEEE}, 70(9):939--952, 1982.

\bibitem{jaynes2003}
Edwin~T. Jaynes.
\newblock {\em Probability Theory: The Logic of Science}.
\newblock Cambridge University Press, Cambridge, U.K., 2003.
\newblock Edited by G. Larry Bretthorst.

\bibitem{keefer-bodily1983}
Donald~L. Keefer and Samuel~E. Bodily.
\newblock Three-point approximations for continuous random variables.
\newblock {\em Management Science}, 29(5):595--609, May 1983.

\bibitem{kitamura-stutzer1997}
Yuichi Kitamura and Michael Stutzer.
\newblock An information-theoretic alternative to generalized method of moments
  estimation.
\newblock {\em Econometrica}, 65(4):861--874, July 1997.

\bibitem{knuth-skilling2012}
Kevin~H. Knuth and John Skilling.
\newblock Foundations of inference.
\newblock {\em Axioms}, 1(1):38--73, 2012.

\bibitem{krusell-smith1998}
Per Krusell and Anthony~A. Smith, Jr.
\newblock Income and wealth heterogeneity in the macroeconomy.
\newblock {\em Journal of Political Economy}, 106(5):867--896, October 1998.

\bibitem{kullback1967}
Solomon Kullback.
\newblock A lower bound for discrimination information in terms of variation.
\newblock {\em IEEE Transactions on Information Theory}, 13(1):126--127,
  January 1967.

\bibitem{kullback-leibler1951}
Solomon Kullback and Richard~A. Leibler.
\newblock On information and sufficiency.
\newblock {\em Annals of Mathematical Statistics}, 22(1):79--86, 1951.

\bibitem{luenberger-ye2008}
David~G. Luenberger and Yinyu Ye.
\newblock {\em Linear and Nonlinear Programming}.
\newblock International Series in Operations Research and Management Science.
  Springer, NY, third edition, 2008.

\bibitem{mead-papanicolaou1984}
Lawrence~R. Mead and Nikos Papanicolaou.
\newblock Maximum entropy in the problem of moments.
\newblock {\em Journal of Mathematical Physics}, 25(8):2404--2417, August 1984.

\bibitem{merton1971}
Robert~C. Merton.
\newblock Optimum consumption and portfolio rules in a continuous-time model.
\newblock {\em Journal of Economic Theory}, 3(4):373--413, December 1971.

\bibitem{miller-rice1983}
Allen~C. Miller, III and Thomas~R. Rice.
\newblock Discrete approximations of probability distributions.
\newblock {\em Management Science}, 29(3):352--362, March 1983.

\bibitem{pflug2001}
Georg~Ch. Pflug.
\newblock Scenario tree generation for multiperiod financial optimization by
  optimal discretization.
\newblock {\em Mathematical Programming}, 89(2):251--271, January 2001.

\bibitem{samuelson1969}
Paul~A. Samuelson.
\newblock Lifetime portfolio selection by dynamic stochastic programming.
\newblock {\em Review of Economics and Statistics}, 51(3):239--246, August
  1969.

\bibitem{shore-johnson1980}
John~E. Shore and Rodney~W. Johnson.
\newblock Axiomatic derivation of the principle of maximum entropy and the
  principle of minimum cross-entropy.
\newblock {\em IEEE Transactions on Information Theory}, 26(1):26--37, January
  1980.

\bibitem{smith1993}
James~E. Smith.
\newblock Moment methods for decision analysis.
\newblock {\em Management Science}, 39(3):340--358, March 1993.

\bibitem{stutzer1995}
Michael Stutzer.
\newblock A {B}ayesian approach to diagnosis of asset pricing models.
\newblock {\em Journal of Econometrics}, 68(2):367--397, August 1995.

\bibitem{stutzer1996}
Michael Stutzer.
\newblock A simple nonparametric approach to derivative security valuation.
\newblock {\em Journal of Finance}, 51(5):1633--1652, December 1996.

\bibitem{tanaka-toda-maxent-approx}
Ken'ichiro Tanaka and Alexis~Akira Toda.
\newblock Discrete approximations of continuous distributions by maximum
  entropy.
\newblock {\em Economics Letters}, 118(3):445--450, March 2013.

\bibitem{tauchen1986-EL}
George Tauchen.
\newblock Finite state {M}arkov-chain approximations to univariate and vector
  autoregressions.
\newblock {\em Economics Letters}, 20(2):177--181, 1986.

\bibitem{tauchen-hussey1991}
George Tauchen and Robert Hussey.
\newblock Quadrature-based methods for obtaining approximate solutions to
  nonlinear asset pricing models.
\newblock {\em Econometrica}, 59(2):371--396, March 1991.

\bibitem{toda2010}
Alexis~Akira Toda.
\newblock Existence of a statistical equilibrium for an economy with endogenous
  offer sets.
\newblock {\em Economic Theory}, 45(3):379--415, 2010.

\bibitem{toda-BGE}
Alexis~Akira Toda.
\newblock {B}ayesian general equilibrium.
\newblock {\em Economic Theory}, 2014.

\bibitem{trefethen2013}
Lloyd~N. Trefethen.
\newblock {\em Approximation Theory and Approximation Practice}.
\newblock SIAM, Philadelphia, 2013.

\bibitem{vancampenhout-cover1981}
Jan~M. Van~Campenhout and Thomas~M. Cover.
\newblock Maximum entropy and conditional probability.
\newblock {\em IEEE Transactions on Information Theory}, 27(4):483--489, July
  1981.

\bibitem{wu2003}
Ximing Wu.
\newblock Calculation of maximum entropy densities with application to income
  distribution.
\newblock {\em Journal of Econometrics}, 115(2):347--354, August 2003.

\end{thebibliography}

\end{document}